\documentclass[reqno, 11pt]{amsart}
\usepackage{amsthm, amsmath, amssymb, graphicx, tikz,enumerate,paralist,bbm}
\usepackage[margin=1in]{geometry}

\usepackage[color,final]{showkeys} 
\usepackage[colorlinks=true, pdfstartview=FitV, 
            pagebackref=false]{hyperref}

\definecolor{refkey}{gray}{.75}
\definecolor{labelkey}{gray}{.5}
\numberwithin{equation}{section}

\newtheorem{theorem}{Theorem}[section]
\newtheorem{proposition}[theorem]{Proposition}
\newtheorem{lemma}[theorem]{Lemma}
\newtheorem{corollary}[theorem]{Corollary}

\theoremstyle{definition}
\newtheorem{definition}[theorem]{Definition}
\newtheorem{example}[theorem]{Example}

\theoremstyle{remark}
\newtheorem*{remark}{Remark}

\newcommand{\abs}[1]{\left\lvert#1\right\rvert}
\newcommand{\norm}[1]{\left\lVert#1\right\rVert}

\newcommand{\ceil}[1]{\left\lceil #1 \right\rceil}
\newcommand{\paren}[1]{\left( #1 \right)}

\newcommand{\set}[1]{\left\{ #1 \right\}}

\newcommand{\cond}[2]{\left( #1 \;\middle\vert\; #2 \right)}

\newcommand{\one}{\mathbbm{1}}
\newcommand{\wt}{\widetilde}

\newcommand{\op}{\mathrm{op}}
\newcommand{\ol}{\overline}
\newcommand{\ul}{\underline}
\newcommand{\uq}{\underline{q}}
\newcommand{\oq}{\overline{q}}
\renewcommand{\epsilon}{\varepsilon}

\newcommand{\x}{\times}

\newcommand{\e}{\epsilon}
\renewcommand{\P}{\mathbb{P}}
\newcommand{\R}{\mathbb{R}}
\newcommand{\cB}{\mathcal{B}}

\newcommand{\cW}{\mathcal{W}}
\newcommand{\cE}{\mathcal{E}}
\newcommand{\cG}{\mathcal{G}}

\newcommand{\Ip}{h_p}
\newcommand{\ER}{\widetilde{\mathcal{C}}}

\tikzstyle{P} = [draw, circle, black, fill, inner sep = 0pt, minimum width = 3pt]
\tikzstyle{every loop} = []
\usetikzlibrary{calc}

\newcommand{\tikzind}{
  \begin{tikzpicture}[baseline, yshift=3pt]
    \path[use as bounding box] (-.35,-.2) rectangle (.6,.2);
    \draw (0.5,0)
  node[P] {} -- (0,0) node[P] {} edge[-,in = 135, out = 225, loop] ();
\end{tikzpicture}
}

\author{Eyal Lubetzky}
\address{E.\ Lubetzky\hfill\break
Microsoft Research\\ One Microsoft Way\\ Redmond, WA 98052-6399.}
\email{eyal@microsoft.com}

\author{Yufei Zhao}
\address{Y.\ Zhao\hfill\break
Department of Mathematics\\ MIT\\ Cambridge, MA 02139-4307.}
\email{yufeiz@math.mit.edu}

\title{On replica symmetry of large deviations in random graphs}

\begin{document}

\begin{abstract}
The following question is due to Chatterjee and Varadhan (2011).
Fix $0<p<r<1$ and take $G\sim \cG(n,p)$, the Erd\H{o}s-R\'enyi random graph with edge density $p$,
 conditioned to have at least as many triangles
as the typical $\cG(n,r)$. Is $G$ close in cut-distance to a typical $\cG(n,r)$?
Via a beautiful new framework for large deviation principles in $\cG(n,p)$, Chatterjee and Varadhan gave bounds on the \emph{replica symmetric} phase, the region of $(p,r)$ where the answer is positive. They further showed that for any small enough $p$ there are at least \emph{two} phase transitions as $r$ varies.

We settle this question by identifying the replica symmetric phase for triangles and more generally for any fixed $d$-regular graph.
By analyzing the variational problem arising from the framework of Chatterjee and Varadhan we show that the replica symmetry phase consists of all $(p,r)$ such that $(r^d,\Ip(r))$ lies on the convex minorant of $x\mapsto \Ip(x^{1/d})$ where $\Ip$ is the rate function of a binomial with parameter $p$.
In particular, the answer for triangles involves $\Ip(\sqrt{x})$ rather than the natural guess of $\Ip(x^{1/3})$ where symmetry was previously known.
Analogous results are obtained for linear hypergraphs as well as the setting where the largest eigenvalue of $G\sim\cG(n,p)$ is conditioned to exceed the typical value of the largest eigenvalue of $\cG(n,r)$.
Building on the work of Chatterjee and Diaconis (2012) we obtain additional results on a class of exponential random graphs including a new range of parameters where symmetry breaking occurs. En route we give a short alternative proof of a graph homomorphism inequality due to Kahn~(2001) and Galvin and Tetali~(2004).
\end{abstract}

\maketitle

\vspace{-0.25in}

\section{Introduction} \label{sec:intro}

The following question was raised by Chatterjee and Varadhan~\cite{CV11} concerning large deviations in $\cG(n,p)$, the Erd\H{o}s-R\'enyi random graph on $n$ vertices with edge density $p$.
\begin{quote}
  Fix $0 < p < r < 1$ and let $G_n$ be an instance of $\cG(n,p)$ conditioned on the rare event of having at least
  as many triangles as a typical instance of $\cG(n,r)$.
  Is it the case that as $n\to\infty$ the graph $G_n$ is close in cut-distance to a typical $\cG(n,r)$ graph?
\end{quote}
(A more formal statement, including the definition of the graph cut-metric, is postponed to \S\ref{sec:intro-phase-diagrams}.)
This amounts to asking whether the likely reason for too many triangles is an overwhelming number of edges, uniformly distributed, or some fewer edges arranged in a special structure, e.g., a clique.
Dubbed \emph{replica symmetry} and \emph{symmetry breaking}, resp., the dichotomy between these scenarios turns out to depend on $p$ and $r$. Intriguingly, it was known that for small enough $p$  there are at least \emph{two} phase transitions as $r$ increases from $p$ to $1$, with symmetry replica near the two endpoints.

In this work we analyze the variational problems arising from the framework of Chatterjee and Varadhan and obtain a full answer for the question above, as depicted in Fig.~\ref{fig:phase}. More generally, we identify the phase diagram for upper tails of any fixed regular subgraph and derive related results in other random graph settings, e.g., exponential random graphs, random hypergraphs, etc.

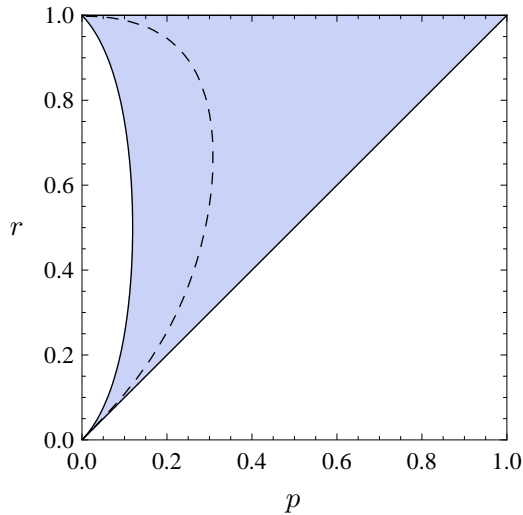
\begin{figure}
  \centering
  \begin{tikzpicture}
    \node[anchor=south west] (plot) at (0,0)
    {\includegraphics{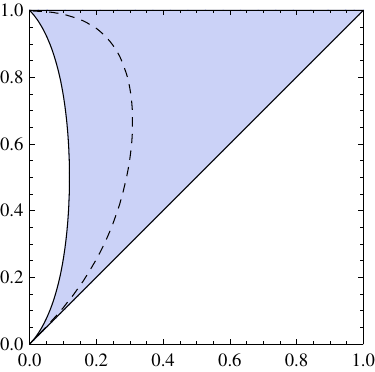}};

    \begin{scope}[x={(plot.south east)},y={(plot.north west)}]
      \path[use as bounding box] (0,-.02) rectangle (1,1);
      \node[below] at (0.52,0) {$p$};
      \node[left] at (0,0.52) {$r$};
    \end{scope}
  \end{tikzpicture}
  \caption{Phase diagram for the upper tail of triangle
    counts. Shaded region is the replica symmetric phase; the
    region to its left is the symmetry breaking phase. Previous
    results~\cite{CD10,CV11} established replica symmetry
    to the right of the dashed curve.}
  \label{fig:phase}
  \vspace{-0.1cm}
\end{figure}

\subsection{Subgraph densities and spectral radii}\label{sec:intro-phase-diagrams}

Large deviations for subgraph
densities in random graphs have been extensively studied (see, e.g.,~\cite{JR02,Vu01,KV04,JOR04,JR04,Cha12,DK12b,DK12} as well as~\cite{Bol,JLR} and the references therein).
A representing example which drew significant attention is upper tails of triangle counts, i.e., estimating the probability that $\cG(n,p)$ has at least $\binom{n}{3}r^3 $ triangles where $r=(1+ \eta)p$ for fixed $\eta>0$ (allowing $p$ to vary with $n$), a problem whose understanding is still incomplete. The order of the rate function (the normalized logarithm of this probability) when $p\to 0$ was only very recently settled: Chatterjee~\cite{Cha12} and DeMarco and Kahn~\cite{DK12} independently established it to be $ n^2p^2 \log (1/p)$
when $p \gtrsim \log n / n$, and yet the exact rate function remains unknown in this range of $p$. We now turn to what was known for fixed $p$, our focus in this paper.

Clearly, if the total number of edges in $\cG(n,p)$ deviates to $m\sim \binom{n}2 r$ then one will arrive at a random graph with $m$ uniformly distributed edges featuring the desired number of triangles. Thus, the large deviation rate function for encountering $\binom{n}3 r^3$ triangles in $\cG(n,p)$ is at most
$\Ip(r)$, where
\begin{equation}
  \label{eq:Ip}
\Ip(x) := x \log \frac{x}{p} + (1-x) \log \frac{1-x}{1-p}\quad\mbox{ for $p \in (0,1)$ and $x \in [0,1]$}
\end{equation}
is the rate function associated to the binomial distribution with probability $p$. 
However, it is possible that other configurations with broken symmetry would give rise to lower rate functions.

As an application of Stein's method for concentration inequalities, Chatterjee and
Dey~\cite{CD10} found a range of $(p,r)$ where the large deviation rate function for triangles
is equal to $\Ip(r)$, namely when $p \geq 2/(2+e^{3/2}) \approx 0.31$ or when $r$ is suitably close either to $p$
or to $1$. This symmetry region was explicitly stated in~\cite[Theorem 4.3]{CV11} as all pairs $(p,r)$
where $(r^3,\Ip(r))$ lies on the convex minorant of $x\mapsto \Ip(x^{1/3})$.
The breakthrough work of Chatterjee and Varadhan~\cite{CV11} introduced a remarkable general framework for large deviation principles in $\cG(n,p)$ via
Szemer\'edi's regularity lemma~\cite{Sze78} and the theory of
graph limits by Lov\'asz et al.~\cite{LS06,LS07,BCLSV08}. It expressed the large deviation rate function, and moreover the \emph{structure} of the random graph conditioned on the large deviation, in terms of a variational problem on \emph{graphons}, the infinite-dimensional limit objects for graph sequences.

Although often this variational problem is untractable, for triangles in the mentioned range of $(p,r$) it was shown in~\cite{CV11} to have a unique and symmetric solution.
To formalize this symmetry, we say a graph $G$ is close in cut-distance to a typical $\cG(n,r)$ graph if all induced subgraphs on a linear number of vertices have edge
density close to $r$. More precisely, for a graph $G$ and $r \in [0,1]$ let
\[
\delta_\square(G, r) := \sup_{A, B \subset V(G)} \frac{1}{\abs{V(G)}^2}
\big|e_G(A, B) - r\abs{A}\abs{B}\big|\,,
\]
where $e_G(A,B)$ is the number of pairs $(a,b)\in A\times B$ with $ab\in E(G)$.
Chatterjee and Varadhan showed that, in the above range of $(p,r)$, if $G_n\sim \cG(n,p)$ is conditioned to have at least $\binom{n}3 r^3$ triangles then $\delta_\square(G_n,r)\to 0$ in probability as $n\to\infty$.
The function $x\mapsto\Ip(x^{1/3})$ governing that region
is the natural candidate for the phase boundary as the cube-root accounts for the 3 edges of the triangle (see, e.g.,~\cite[\S4.5.2]{DZ} for related literature), and indeed Chatterjee and Varadhan asked whether this precisely characterizes the full replica symmetric phase.
As it turns out, however, the replica symmetric phase is strictly larger, being governed instead by $x\mapsto\Ip(\sqrt{x})$. (See Fig.~\ref{fig:phase}.)

\begin{figure}
  \centering
  \begin{tikzpicture}
    \node[anchor=south west] (plot) at (0,0)
    {\includegraphics{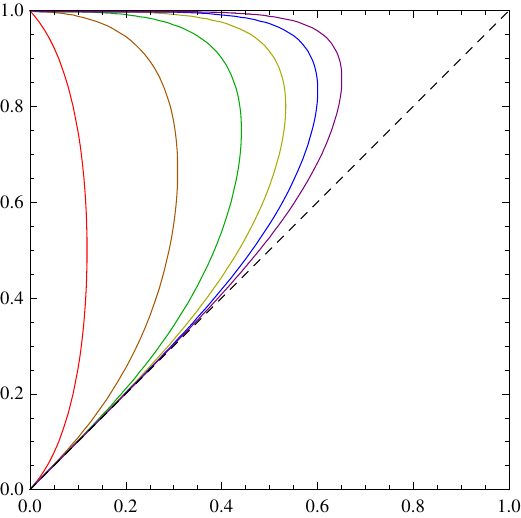}};
    \begin{scope}[x={(plot.south east)},y={(plot.north west)}]
      \path[use as bounding box] (0,-.02) rectangle (1,1);
      \node[below] at (0.52,0) {$p$};
      \node[left] at (0,0.52) {$r$};
      \node at (.12,.55) {\color{red}{\footnotesize$d=2$}};
      \node at (.28,.65) {\color{orange!67!black}{\footnotesize$d=3$}};
      \node at (.41, .7) {\color{green!67!black}{\footnotesize$d=4$}};
      \node at (.54, .77) {\color{yellow!67!black}{\footnotesize$d=5$}};
      \node at (.6, .83) {\color{blue}{\footnotesize$d=6$}};
      \node at (.7, .89) {\color{purple}{\footnotesize$d=7$}};
    \end{scope}
  \end{tikzpicture}
  \caption{The phase boundary for counts of $d$-regular fixed subgraphs in $\cG(n,p)$.}
  \label{fig:phase-d}
    \vspace{-0.2cm}
\end{figure}
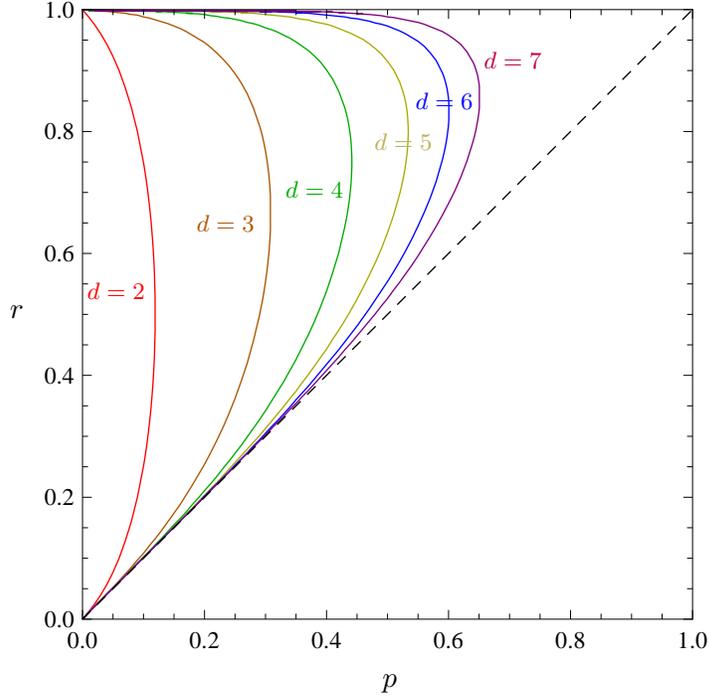

For any graph $G$ let $e(G) = \abs{E(G)}$, and for any two graphs $G$ and $H$ let $\hom(H, G)$ denote the number of
homomorphisms from $H$ to $G$ (i.e., maps $V(H) \to V(G)$ that carry
edges to edges). Let
\[
t(H,G) := \frac{\abs{\hom(H,G)}}{\abs{V(G)}^{\abs{V(H)}}}
\]
be the 
probability that a random map $V(H) \to V(G)$ is a graph homomorphism.
We now state our main result on the phase diagram for large deviations in densities of $d$-regular subgraphs.

{\samepage
\begin{theorem}
  \label{thm:count-ldp}
  Fix $0 < p \leq r <
  1$ and let $H$ be a fixed $d$-regular graph for some $d\geq 2$. Let $G_n\sim\cG(n,p)$ be the
  Erd\H{o}s-R\'enyi random graph on $n$ vertices with edge probability
  $p$.
  \begin{enumerate}[(i)]
  \item\label{it-thm1-1} If the point $(r^d,\Ip(r))$
  lies on the convex minorant of the function $x \mapsto
  \Ip(x^{1/d})$ then
  \[
  \lim_{n \to \infty} \frac{1}{\binom{n}{2}} \log \P\paren{t(H, G_n)
    \geq r^{e(H)} } = - \Ip(r)
  \]
  and furthermore, for every $\epsilon > 0$ there exists some $C=C(H,\e, p, r)>0$ such that for all $n$,
  \[
  \P \cond{\delta_\square(G_n, r) < \e}{t(H, G_n) \geq
    r^{e(H)}} \geq 1 - e^{-C n^2}\,.
  \]
\item\label{it-thm1-2} If the point $(r^d,\Ip(r))$ does not lie on the convex
  minorant of the function $x \mapsto \Ip(x^{1/d})$ then
  \[
  \lim_{n \to \infty} \frac{1}{\binom{n}{2}} \log \P\paren{t(H, G_n)
    \geq r^{e(H)} } > - \Ip(r)
  \]
  and furthermore, there exist $\e, C > 0$ such that for all $n$,
  \[
  \P \cond{
  \inf\big\{  \delta_\square(G_n,s) : 0\leq s \leq 1\big\} > \e}{t(H, G_n) \geq
    r^{e(H)}} \geq 1 - e^{-C n^2} \, .
  \]
  \end{enumerate}
 In particular, when $d = 2$, case \eqref{it-thm1-2} occurs if and only if $p < \left[1 + (r^{-1} - 1)^{1/(1-2r)}\right]^{-1}$.
\end{theorem}

The boundary curves for various values of $d$ are plotted in
Fig.~\ref{fig:phase-d}. It is easy to verify (Lemma~\ref{lem:hp-shape}) that the rightmost point
in the curve for $d$-regular subgraphs is $(p,r) = \big(\frac{d-1}{d-1
  + e^{d/(d-1)}}, \frac{d - 1}d\big)$.
}

We give an analogous result for large deviations of the spectral radius
of an Erd\H{o}s-R\'enyi random graph. The phase boundary in this case coincides
with that of triangles.

\begin{theorem}
  \label{thm:spectral-ldp}
  Fix $0 < p \leq r < 1$. Let $G_n\sim \cG(n,p)$ be an Erd\H{o}s-R\'enyi random graph on $n$
  vertices with edge probability $p$, and let $\lambda_1(G_n)$ denote the largest
  eigenvalue of its adjacency matrix.
  \begin{enumerate}[(i)]
  \item\label{it-thm2-1} If $p \geq \left[1 + (r^{-1} - 1)^{1/(1-2r)}\right]^{-1}$ then
  \[
  \lim_{n \to \infty} \frac{1}{\binom{n}{2}} \log
  \P\paren{\lambda_1(G_n)\geq n r
    } = - \Ip(r)
  \]
  and furthermore, for every $\epsilon > 0$ there exists some $C=C(\e, p, r)>0$ such that for all $n$,
  \[
  \P \cond{\delta_\square(G_n, r) < \e}{\lambda_1(G_n)\geq n r} \geq 1 - e^{-C n^2}\,.
  \]
\item\label{it-thm2-2} If $p < \left[1 + (r^{-1} - 1)^{1/(1-2r)}\right]^{-1}$ then
  \[
  \lim_{n \to \infty} \frac{1}{\binom{n}{2}} \log \P\paren{\lambda_1(G_n)\geq n r} > - \Ip(r)
  \]
  and furthermore, there exist $\e, C > 0$ such that for all $n$,
  \[
  \P \cond{\inf\big\{  \delta_\square(G_n,s) : 0\leq s \leq 1\big\} > \e}{\lambda_1(G_n)\geq
    n r} \geq 1 - e^{-C n^2}.
  \]
  \end{enumerate}
\end{theorem}

Both theorems are proved through an analysis of the graphon variational problems rising from the framework of Chatterjee and
Varadhan~\cite{CV11}. We show that throughout the replica symmetric region its unique solution is the symmetric one (a consequence of a generalized form of H\"older's inequality), whereas elsewhere one can construct graphons that outperform the symmetric candidate. Note that Theorem~\ref{thm:spectral-ldp} addresses spectral large deviations, whereas the framework of~\cite{CV11} was tailored for subgraph densities (the recent work~\cite{CV12} broadens it to general random matrix properties with respect to an appropriately defined spectral distance. Here we consider concretely large deviations in the spectral norm of random graphs). Fortunately, the results of~\cite{CV11} easily extend to a wide family of graph parameters with respect to the cut-metric,
including the operator norm, the extension of the (normalized) spectral norm to the space of graphons.
This generalization is detailed in~\S\ref{sec:CV}.

{
\begin{figure}
\setlength{\abovecaptionskip}{-5pt}
  \centering
  \begin{tikzpicture}[font=\footnotesize]

    \newcommand{\hsep}{5.6cm}
    \newcommand{\vsep}{5.8cm}
    \node (plot1) at (0,0) {
      \includegraphics[scale=.85]{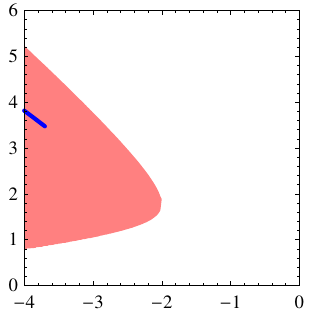}};
    \node (plot2) at (\hsep,0) {
      \includegraphics[scale=.85]{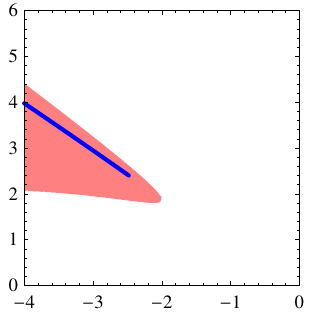}};
    \node (plot3) at (2*\hsep,0) {
      \includegraphics[scale=.85]{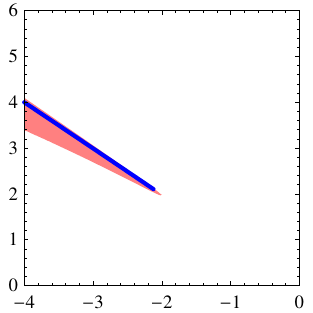}};
    \node (plot4) at (0,-\vsep) {
      \includegraphics[scale=.85]{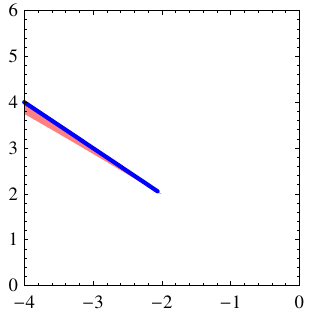}};
    \node (plot5) at (\hsep,-\vsep) {
      \includegraphics[scale=.85]{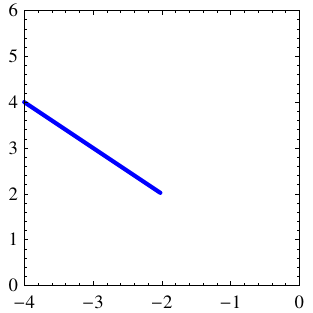}};
    \node (plot6) at (2*\hsep,-\vsep) {
      \includegraphics[scale=.85]{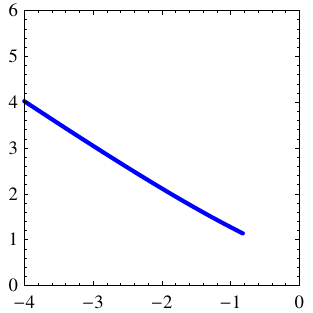}};

    \path[use as bounding box] ($(plot1.north west)+(-.5cm,1)$) rectangle
    ($(plot6.south east)+(0,-.5)$);

    \begin{scope}[shift={(plot1.south west)},x={(plot1.south east)},y={(plot1.north west)}]
      \node[below=-.5em] at (.5,0) {$\beta_1$};
      \node[left=-.7em] at (0,.5) {$\beta_2$};
      \node[above] at (.5,1) {$\alpha=0.5$};
      \draw[dashed] (.515,0.11)--(.515,.95);
      \draw (.8,.9) node[align=left,anchor=north east,fill=white,font=\tiny]
      {Broken symmetry in \\ the shaded region} edge[-latex] (.3,.5);
    \end{scope}

    \begin{scope}[shift={(plot2.south west)},x={(plot2.south east)},y={(plot2.north west)}]
      \node[below=-.5em] at (.5,0) {$\beta_1$};
      \node[left=-.7em] at (0,.5) {$\beta_2$};
      \node[above] at (.5,1) {$\alpha=0.6$};
      \draw[dashed] (.515,0.11)--(.515,.95);
      \node[align=left,font=\tiny,text width=6em] at (.8,.5) {
        Replica \\ symmetric \\
        for $\beta_1 \geq -2$.};         
    \end{scope}

    \begin{scope}[shift={(plot3.south west)},x={(plot3.south east)},y={(plot3.north west)}]
      \node[below=-.5em] at (.5,0) {$\beta_1$};
      \node[left=-.7em] at (0,.5) {$\beta_2$};
      \node[above] at (.5,1) {$\alpha=0.65$};
      \draw[dashed] (.515,0.11)--(.515,.95);
    \end{scope}

    \begin{scope}[shift={(plot4.south west)},x={(plot4.south east)},y={(plot4.north west)}]
      \node[below=-.5em] at (.5,0) {$\beta_1$};
      \node[left=-.7em] at (0,.5) {$\beta_2$};
      \node[above] at (.5,1) {$\alpha=0.66$};
          \draw[dashed] (.515,0.11)--(.515,.95);
    \end{scope}

    \begin{scope}[shift={(plot5.south west)},x={(plot5.south east)},y={(plot5.north west)}]
      \node[below=-.5em] at (.5,0) {$\beta_1$};
      \node[left=-.7em] at (0,.5) {$\beta_2$};
      \node[above] at (.5,1) {$\alpha=2/3$};
      \draw (.7,.8) node[align=left,font=\tiny,text width=9em]
      {Replica symmetric everywhere};
    \end{scope}

    \begin{scope}[shift={(plot6.south west)},x={(plot6.south east)},y={(plot6.north west)}]
      \node[below=-.5em] at (.5,0) {$\beta_1$};
      \node[left=-.7em] at (0,.5) {$\beta_2$};
      \node[above] at (.5,1) {$\alpha=1$};
      \draw (.9,.9) node[anchor=north east,align=left,font=\tiny,fill=white]
      {$u^*(\beta_1,\beta_2)$ is discontinuous \\across the curve $\Gamma$} edge[-latex] (.33,.52);
    \end{scope}
  \end{tikzpicture}
  \caption{
    The $(\beta_1,\beta_2)$-phase diagrams for the exponential
    random graph model in~\eqref{eq-exp-model-alpha} 
    with $\beta_2\geq 0$ and various values of $\alpha>0$,
    as a special case of Theorem~\ref{thm:exp}.
    When $\alpha < 2/3$, symmetry breaking occurs in the shaded region (at least)
    and replica symmetry occurs for $\beta_1 \geq -2$.
    When $\alpha \geq 2/3$, replica symmetry always occurs.}
  \label{fig:betaphase}
  \vspace{-0.2cm}
\end{figure}
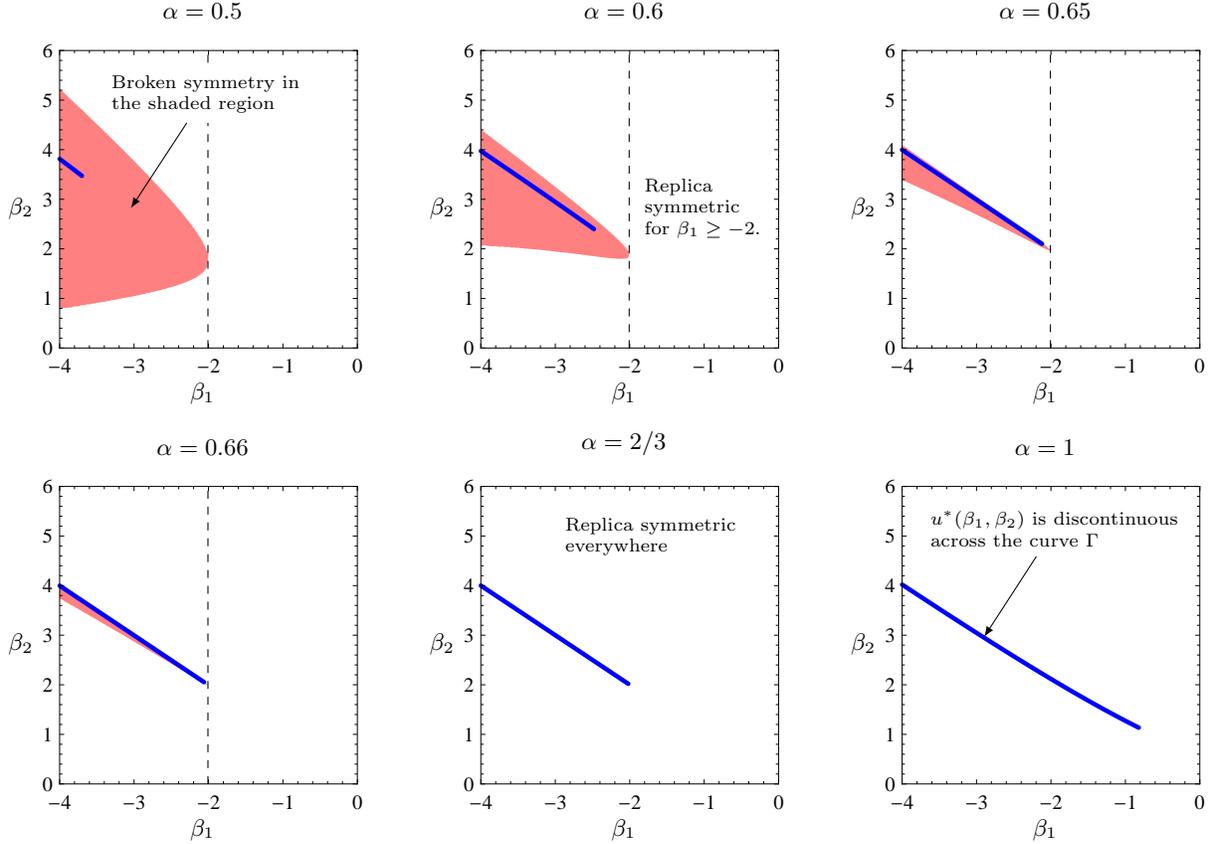
}

\subsection{Exponential random graphs}
We now turn our attention to a different random graph model, the basic setting of which assigns a probability $p_\beta(G)$ to every graph $G$ on $n$ labeled vertices as a function of its edge density $t(K_2,G)$, its triangles density $t(K_3,G)$ and a weight vector $\beta=(\beta_1,\beta_2)$ for these two quantities.\footnote{The general model allows for an arbitrary (fixed) collection of subgraphs. While the majority of our arguments can be extended to the general setting, we
  focus on the two-term model for simplicity and clarity.} Namely, the graph $G$ appears with the following
probability:
\begin{equation}
  \label{eq-exp-model}
p_\beta(G) = \frac{1}{Z_n} \exp\paren{\binom{n}{2}(\beta_1 t(K_2, G) + \beta_2 t(K_3, G))}\,,
\end{equation}
where $Z_n$ is a normalizing factor (the partition function).
When $\beta_2 > 0$ the model
favors graphs with more triangles whereas triangles are discouraged for $\beta_2<0$. There is a rich literature on both flavors of the model, motivated in part by applications in social networking: the reader is referred to~\cite{HJ99,PN04,PN05,Strauss86} as well as~\cite{BBS11,CD} and the references therein.

As shown by Bhamidi, Bresler and Sly~\cite{BBS11} and Chatterjee and Diaconis~\cite{CD},
when $\beta_2 \geq 0$ and $n$ is large, a
typical random graph drawn from the distribution has a trivial structure --- essentially the
same one as an Erd\H{o}s-R\'enyi random graph with a suitable edge density. This
somewhat disappointing conclusion accounts for some of the practical
difficulties with statistical parameter estimation for such models. It
was further shown in \cite{CD} that if we allow $\beta_2$ to be sufficiently
negative, then the model does behave appreciably differentially from
an Erd\H{o}s-R\'enyi model. In this part of our work we focus on the case
$\beta_2 > 0$, and propose a natural generalization that will enable the model
to exhibit a nontrivial structure instead of the previously observed Erd\H{o}s-R\'enyi
behavior.

Consider the exponential random graph model which includes an additional exponent $\alpha>0$
in the exponent of the triangle density term:
\begin{equation}
  \label{eq-exp-model-alpha}
  p_{\alpha,\beta}(G) = \frac{1}{Z_n} \exp\paren{\binom{n}{2}(\beta_1 t(K_2, G) + \beta_2 t(K_3, G)^\alpha)}
\, .
\end{equation}
We will show that this model exhibits a symmetry breaking phase transition even when $\beta_2 > 0$.
When $\alpha \geq 2/3$, the generalized model features the Erd\H{o}s-R\'enyi behavior, similar to the
previously observed case of $\alpha = 1$. However, for $0< \alpha < 2/3$, there exist regions of values of
$(\beta_1, \beta_2)$ for which a typical random graph drawn from this
distribution has symmetry breaking. As was the case for Theorems~\ref{thm:count-ldp} and \ref{thm:spectral-ldp}, rather than just triangles we prove this result for any $d$-regular graph $H$.

\begin{theorem}
  \label{thm:exp}
    Let $H$ be a $d$-regular graph for some fixed $d\geq 2$ and fix $\beta_1 \in \R$ and $\beta_2,\alpha > 0$.
  Let be $G_n$ be an exponential random graph on $n$ labeled vertices
  with law
\begin{equation}
  \label{eq-exp-model-H-alpha}
  p_{\alpha,\beta}(G_n) = \frac{1}{Z_n} \exp\paren{\binom{n}{2}(\beta_1 t(K_2, G_n) + \beta_2 t(H, G_n)^\alpha)}
\, .
\end{equation}
  \begin{enumerate}[(a)]
  \item Suppose $\alpha \geq d/e(H)$. There exists a subset $\Gamma =
    \{(\beta_1, \varphi(\beta_1)) : \beta_1 < \log(e(H)\alpha - 1) -
    \tfrac{e(H)\alpha}{e(H)\alpha - 1}\}$ of $\R^2$ for some function $\varphi:\R\to\R$
  such that for every
    $(\beta_1, \beta_2) \in \R \x (0,\infty) \setminus \Gamma$
   there exists $0< u^* < 1$ so that $\delta_\square(G_n,u^*)\to 0$ almost surely as $n\to\infty$,
   and for every $(\beta_1,\beta_2)\in\Gamma$ there exist $0 < u^*_1 < u^*_2 <
    1$ such that $\min\{ \delta_\square(G_n, u^*_1) \;,\; \delta_\square(G_n, u^*_2)\}\to 0$ almost surely as $n\to\infty$.
  \item Suppose $0 < \alpha < d/e(H)$ and $\beta_1 \geq \log(d-1) -
    d/(d-1)$. Then there
    exists $0< u^* < 1$ such that for every $\e > 0$ there is a $C > 0$
    such that $\delta_\square(G_n, u^*)\to 0$ almost surely as $n\to\infty$.
  \item\label{it:thm-exp-c} Suppose $0 < \alpha < d/e(H)$ and $\beta_1 < \log(d-1) -
    d/(d-1)$.
    Then there exists an open interval of values $\beta_2 > 0$ with
    the property that there exist $\e, C > 0$ such that for all $n$,
    \[
    \P \paren{
    \inf\big\{  \delta_\square(G_n,s) : 0\leq s \leq 1\big\}  > \e} \geq 1 - e^{-C n^2} \, .
    \]
  \end{enumerate}
\end{theorem}
Note that, as in the previous theorems, for the replica symmetric phase one can quantify the rate of convergence, e.g., when $\delta_\square(G_n,u^*)\to 0 $ almost surely we in fact have that for any $\epsilon>0$ there exists some $C>0$ so that $\P(\delta_\square(G_n,u^*)\leq \epsilon) \geq 1-e^{-Cn^2}$ for every $n$.

\subsection{Linear hypergraphs}
Theorem~\ref{thm:linear-hypergraph-ldp} (see \S\ref{sec:hypergraphs}) extends Theorem~\ref{thm:count-ldp} to the setting of random hypergraphs. A $k$-uniform
hypergraph $G$ consists of a set $V(G)$ of vertices and a set $E(G)$
of hyperedges, where each hyperedges is a $k$-element subsets of
$V(G)$. It is said to be $d$-regular if every vertex is incident to exactly $d$ edges, and
\emph{linear} if every two vertices are incident to at most one common hyperedge (see Fig.~\ref{fig:linear-hyp} for examples of $d$-regular 3-uniform linear hypergraphs).
The random hypergraph $\cG^{(k)}(n,p)$ is formed by starting with
$n$ vertices and adding $k$-element subset of the vertices as a
hyperedge independently with probability $p$.

In order to generalize our arguments to large deviations in the density of $H$, an arbitrary $d$-regular linear hypergraph, one must first extend the theory developed by Chatterjee and Varadhan~\cite{CV11} to $k$-uniform hypergraphs.
Thanks to the linearity of the hypergraph $H$ there is a simple extension of Szemer\'edi's regularity lemma to hypergraphs that behaves well with respect to the density of $H$.

\begin{figure}
 \includegraphics[scale=.38]{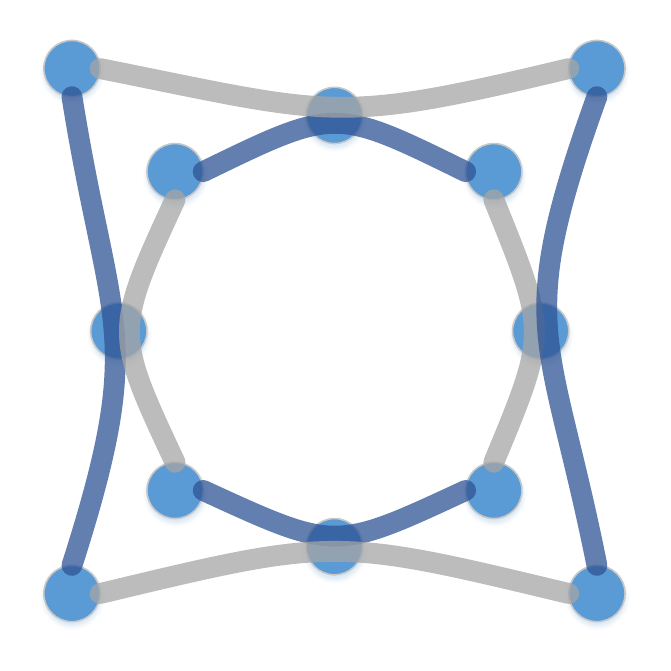}
 \hspace{1in}
 \includegraphics[scale=.4]{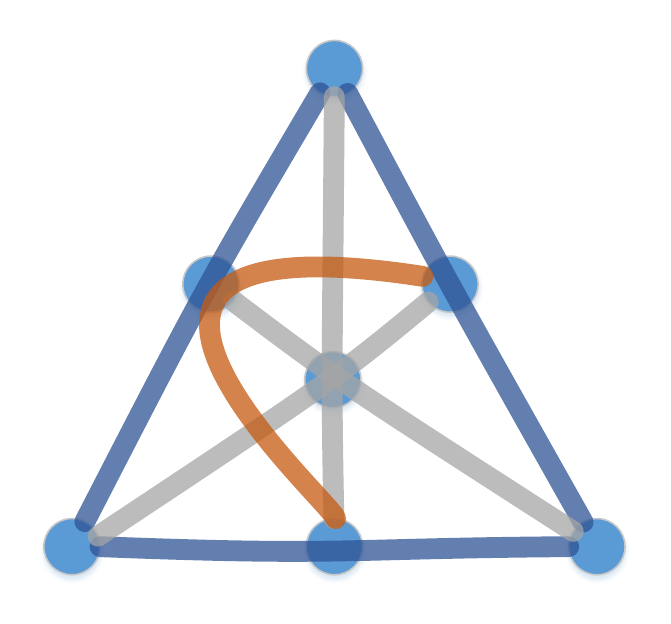}
  \caption{Linear $d$-regular hypergraphs: a cycle ($d=2$) and the Fano plane ($d=3$).}
  \label{fig:linear-hyp}
\end{figure}

\subsection{Graph homomorphisms}
Alon~\cite{Alon91} conjectured in 1991 that the number of independent sets in a $d$-regular graph $G$, denoted $i(G)$, satisfies $i(G) \leq i(K_{d,d})^{|V(G)|/(2d)}$, i.e., it is maximized when $G$
is a union of complete bipartite graphs $K_{d,d}$.
Kahn~\cite{Kahn01} verified this when $G$ is bipartite using an ingenious application of the entropy method
(specifically, Shearer's inequality).
This result was thereafter extended by the second author~\cite{Zhao10} to all $d$-regular graphs via an
elementary bijection. Using the entropy method of Kahn, Galvin
and Tetali~\cite{GT04} extended~\cite{Kahn01} to graph
homomorphisms:
\begin{theorem}[Galvin and Tetali~\cite{GT04}, following Kahn~\cite{Kahn01}]
  \label{thm:GT}
Let $G$ be a simple $d$-regular bipartite graph, and let $H$ be a graph, possibly containing
loops. We have
\begin{equation} \label{eq:GT-intro}
\hom(G, H) \leq \hom(K_{d,d}, H)^{\abs{V(G)}/(2d)}.
\end{equation}
\end{theorem}
Observe that this inequality generalizes the independent set result, since $i(G) =
\hom(G, \tikzind)$.
Previously, all the known proofs of these inequalities relied on entropy techniques.
Regarding a more elementary proof, Kahn~\cite{Kahn01} wrote that
``one would think that this simple and natural conjecture... would have a simple and natural proof.''
As a related aside, in \S\ref{sec:hom} we give a short new entropy-free proof for \eqref{eq:GT-intro} as an
immediate consequence of the generalized form of H\"older's inequality.

\subsection{Organization}
In \S\ref{sec:CV} we review graph limits
as well as
 the large deviation principle for random graphs developed by
 Chatterjee and Varadhan. In \S\ref{sec:phase} we apply the
 machinery of Chatterjee and Varadhan to prove Theorems~\ref{thm:count-ldp} and~\ref{thm:spectral-ldp},
 determining the phase diagram for
 large deviations of
 subgraph densities and the largest eigenvalue in $\cG(n,p)$, resp.
 Section~\ref{sec:exp} focuses on exponential random graphs and gives the proof of Theorem~\ref{thm:exp}.
 In \S\ref{sec:hypergraphs} we extend Theorem~\ref{thm:count-ldp} to densities of
 linear hypergraphs in random hypergraphs. Section~\ref{sec:hom} contains the short new proof
 of the inequalities of Kahn and Galvin-Tetali (Theorem~\ref{thm:GT}). Finally, in \S\ref{sec:end} we discuss some
 open problems.

\section{Graph limits and the framework of Chatterjee-Varadhan}
\label{sec:CV}

The theory of Chatterjee-Varadhan reduces the problem of determining the
rate function for large deviations in dense random graphs to solving a prescribed variational
problem in graph limits.
We will review the required definitions from graph limit theory and then describe the results of~\cite{CV11} in the broader context of ``nice'' graph parameters, generalizing subgraph counts.

\subsection{Graph limits}\label{sec:graph-limits}

Let $\cW$ be the space of all bounded measurable functions
$[0,1]^2 \to \R$ that are symmetric (i.e., $f(x,y) = f(y,x)$ for all $x,y \in
[0,1]$). Further let $\cW_0$ denote all symmetric measurable functions
$[0,1]^2 \to [0,1]$, referred to as \emph{graphons} or \emph{kernels} (occasionally these are called \emph{labeled graphons} since later we consider
equivalence classes of $\cW_0$ modulo measure-preserving bijections on $[0,1]$).
Lov\'asz and Szegedy~\cite{LS06} showed that the elements of $\cW_0$ are limit objects for sequences of graphs
w.r.t.\ all subgraph densities. Specifically, for any $f \in
\cW$ and any simple graph $H$ with $V(H) = [m] = \set{1,2,\dots,m}$,
define
\[
t(H, f) = \int_{[0,1]^m} \prod_{(i,j) \in E(H)} f(x_i, x_j) \
dx_1 \cdots dx_m\,.
\]
(We shall omit the domain of integration when there is no ambiguity.)
Any simple graph $G$ on vertices $\set{1, 2, \dots, n}$ can be
represented as a graphon $f^G$ by
\[
f^G(x,y) = \begin{cases}
                 1 & \text{if $(\ceil{nx}, \ceil{ny})$ is an edge of $G$}, \\
                 0 & \text{otherwise.}
               \end{cases}
\]
In particular, $t(H, G) = t(H, f^G)$ for any two graphs $H$ and $G$.

A sequence of graphs $\set{G_n}_{n \geq 1}$ is said to converge if
the sequence of subgraph densities $t(H, G_n)$ converges for every fixed finite
simple graph $H$. It was shown in~\cite{LS06} that for any
such convergent graph sequence there is a limit object $f \in \cW_0$ such
that $t(H, G_n) \to t(H, f)$ for every fixed $H$. Conversely, any
$f \in \cW_0$ arises as a limit of a convergent graph sequence.

We will consider several norms on $\cW$, beginning with the
standard $L^p$ norm
\[
\norm{f}_p := \paren{\int \abs{f(x,y)}^p \ dxdy}^{1/p}\,.
\]
Each $f \in \cW$ can be viewed as a Hilbert-Schmidt kernel
operator $T_f$ on $L^2([0,1])$ by
\[
(T_f u)(x) = \int_0^1 f(x,y) u(y) \ dy \quad \mbox{for any $u \in L^2([0,1])$,}
\]
and the operator norm for $f$ is then given by
\[
\norm{f}_\op := \min\set{c \geq 0 : \norm{T_f u}_2 \leq c \norm{u}_2
  \text{ for all } u \in L^2([0,1])}\,.
\]
As $T_f$ is self-adjoint, its operator norm is equal to its spectral radius (see, e.g.,~\cite[Thm.~12.25]{Rudin}).

The cut-norm on $\cW$ is given by
\begin{align*}
\norm{f}_\square &:= \sup_{S, T \subset [0,1]} \abs{\int_{S \x T}
  f(x,y)  \ dxdy}
\\
&= \sup_{u,v \colon [0,1] \to [0,1]} \abs{\int_{[0,1]^2} f(x,y)u(x)v(y)
  \ dxdy},
\end{align*}
where the two suprema are equal
since one only needs to consider $\{0,1\}$-valued $u$ and $v$ by the
linearity of the integral. The second definition is useful for giving
upper bounds using the cut-norm.

For any measure-preserving map $\sigma \colon [0,1] \to [0,1]$ and $f
\in \cW$, define $f^\sigma \in \cW$ to be given by $f^\sigma(x,y) = f(\sigma (x), \sigma
(y))$. We define the cut-distance on $\cW$ by
\[
\delta_\square (f, g) = \inf_{\sigma} \norm{f - g^\sigma}_\square
\]
where $\sigma$ ranges over all measure-preserving bijections on
$[0,1]$.
For the case of graphons this gives a pseudometric space $(\cW_0,
\delta_\square)$, 
which can be turned into a genuine metric
space $\wt \cW_0$, equipped with the same cut-metric,
by taking a quotient w.r.t.\ the equivalence relation $f\sim g$ iff $\delta_\square(f,g) = 0$.
The following theorem can be viewed as a topological interpretation of Szemer\'edi's regularity
lemma.
\begin{theorem}[\cite{LS06}]
  The metric space $(\wt \cW_0, \delta_\square)$ is compact.
\end{theorem}
It was shown in \cite{BCLSV08} that a sequence of graphs
$\set{G_n}_{n \geq 1}$ converges in the sense of subgraph densities if
and only if the sequence of graphons $f^{G_n} \in \cW_0$ converge in
$\cW_0$ w.r.t.\ the cut-distance. Equivalently, the topology on
$\cW_0$ induced by $\delta_\square$ is the weakest topology that is
continuous w.r.t.\ the subgraph densities $t(H, \cdot)$
for every $H$. This underlines one of the reasons making the cut-metric
topology a natural choice for the space of graphons.

\subsection{Graph parameters in the cut-metric topology}

We shall focus on graph parameters whose extensions to the space of graphons
behave well under the cut-metric topology. One example of such a graph parameter is the
subgraph density $t(H, \cdot)$ for an arbitrary finite simple graph
$H$, which was defined in \S\ref{sec:graph-limits} directly on the full space of
graphons such that $t(H,G)=t(H,f^G)$ for any graph $G$.
A crucial feature of $t(H,
\cdot)$ is being continuous w.r.t.\ the cut-metric (\cite{BCLSV08}),
related to the existence of
a ``counting lemma'' in the regularity lemma literature. This will
be a prerequisite for applying the large deviations machinery of Chatterjee and Varadhan.

\begin{definition}\label{def:nice}
  A \emph{nice graph parameter} is a function $\tau \colon \wt\cW_0 \to \R$
  that is continuous w.r.t.\ $\delta_\square$ and such
  that every local extrema of $\tau$ w.r.t.\ $L^\infty(\cW_0)$ is necessarily a global extrema.
  We extend such a function $\tau$ to $\cW_0$ in the obvious manner and further write $\tau(G) = \tau(f^G)$ for any graph $G$.
\end{definition}

Another way to state the local extrema condition is that if $f \in \cW_0$ is not a global
maximum (resp.~minimum) of $\tau$ then for every $\epsilon > 0$ there
exists $g \in \cW_0$ with $\norm{f-g}_\infty < \epsilon$ and
$\tau(g) > \tau(f)$ (resp.~$\tau(g) < \tau(f)$). This
technical condition will later imply the continuity of the rate
function.

Since the metric space $(\wt \cW_0,\delta_\square)$ is compact and path-connected, the image of $\tau$ as above is
a finite closed interval. In particular, its maximum is attained by
a non-empty closed subset of $\wt \cW_0$.

\begin{example}[Subgraph density]
  \label{ex:subgraph}
  For any fixed finite simple graph $H$, the subgraph density $t(H,
  \cdot)$ is a nice graph parameter. As mentioned above,
  $t(H,\cdot)$ is continuous w.r.t.\ $\delta_\square$ and in fact the map
  $f\mapsto t(H,f)$ is Lipschitz-continuous in the metric $\delta_\square$
  (\cite[Theorem~3.7]{BCLSV08}).
  The local extrema condition is fulfilled since the function $g^+=\min\set{f + \e, 1}$ satisfies $t(H, g^+) > t(H,f)$ unless $t(H,f)=1$, and
  similarly $g^-=(1-\epsilon)f$ has $t(H, g^-) < t(H,
  f)$ unless $t(H, f) = 0$.
\end{example}

The next two examples are of graph parameters that do not meet the criteria of Definition~\ref{def:nice}.
\begin{example}[Frobenius norm]
Let $\tau$ be the function that maps a weighted graph $G$ on $n$ vertices with adjacency matrix $A_G$ to the normalized Frobenius norm $\norm{A_G}_\textrm{F}/n$.
Then $\tau$ is discontinuous in $(\wt \cW_0,\delta_\square)$ and therefore is not nice.
Indeed, fix $0 < p < 1$ and let $G_n \sim \cG(n,p)$.
The sequence $G_n$ is known to converge in $\wt\cW_0$ almost surely to the constant graphon $p$ (see~\cite[Theorem~4.5]{BCLSV08}), whose Frobenius norm is $p$.
   In contrast to this limiting value, $\norm{A_{G_n}}_\textrm{F}/n = \norm{f^{G_n}}_2\to \sqrt{p}$ almost surely.
\end{example}

\begin{example}[Max-cut]
  The function $\tau(G) = \mathrm{maxcut}(G)/\abs{V(G)}^2$
  extends to
  $\wt\cW_0$ via \[ \tau(f) = \sup_{U \subset [0,1]} \int_{U \x
    ([0,1] \setminus U)} f(x,y) \ dxdy\,.\]
  Despite being continuous w.r.t.\ $\delta_\square$ as well as monotone,
  the max-cut density is not nice.
  The continuity of $\tau$ follows from the fact $\abs{\tau(f) - \tau(g)} \leq
  \delta_\square(f,g)$, as any cut for
  either $f$ or $g$ translates into a cut for the other with value
  differing by at most $\delta_\square(f,g)$. To see that $\tau$ does not satisfy the local maxima
  condition, let $f$ be the graphon defined to be $1$ on $[0,\frac13] \x
  [\frac13,1] \cup [\frac13,1] \x [0,\frac13]$ and 0 elsewhere. We have $\tau(f) =
  \frac29$, induced by $U = [0,\frac13]$. This is not the global maximum for
  $\tau$, which is $\frac14$ for the constant function $1$. However, we
  claim that $f$ is a local maximizer of $\tau$ with respect to the
  $L^\infty$-topology. By monotonicity,
  showing that $\tau(g) = \frac29$ for the function $g = \min\set{f+ \frac12,1}$
  will imply that $\tau(f_0) \leq \tau(g) =
  \tau(f)$ for any $f_0$ with $\norm{f_0 - f}_\infty \leq \frac12$.
  Indeed, if $\mu(U \cap[0,\frac13]) = a$ and $\mu(U
  \cap [\frac13,1]) = b$, where $\mu$ is Lebesgue measure, then the
  cut density induced by $U$ for $g$ is equal to $\frac{1}{2}a(\frac13
  - a) + \frac12 b(\frac23 - b) + a(\frac23 - b) + b(\frac13 - a)$,
  which is maximized 
  at $(a,b) = (0,\frac23)$ and $(a,b)=(\frac13,0)$.
\end{example}

We will see in \S\ref{sec:phase-eigen} (see Lemma~\ref{lem:eigen-op}) that, in contrast to the Frobenius norm,
the spectral norm (the focus of Theorem~\ref{thm:spectral-ldp}) does behave well under the cut-metric topology, thus qualifying for an application of the large deviation theory of Chatterjee and Varadhan.

\subsection{Large deviations for random graphs}

A random graph $G_n \sim \cG(n,p)$ corresponds to
the random point $f^{G_n} \in  \wt \cW_0$, thus $\cG(n,p)$ induces a
probability distribution $\P_{n,p}$ on $\wt \cW_0$ supported on a
finite set of points (graphs on $n$ vertices).
Recalling~\eqref{eq:Ip}, we extend $\Ip \colon [0,1] \to \R$
to $\cW_0$ by
\[
\Ip(f) := \int_{[0,1]^2} \Ip(f(x,y)) \ dxdy \qquad \mbox{for any $f \in \cW_0$.}
\]
An important feature of $\Ip$ is that it is a
convex function on $[0,1]$ and hence lower-semicontinuous on
$\wt \cW_0$ with respect to the cut-metric topology (\cite[Lem.~2.1]{CV11}).

Using Szemer\'edi's regularity lemma as well as tools from graph limits, Chatterjee and Varadhan~\cite{CV11} proved the following large
deviation principle for random graphs.

\begin{theorem}[\cite{CV11}]
  \label{thm:ldp}
For each fixed $p \in (0,1)$, the sequence $\P_{n,p}$ obeys a
large deviation principle in the space $(\wt \cW_0, \delta_\square)$ with rate function $\Ip$. Explicitly, for any closed set $F \subseteq \wt \cW_0$,
\[
\limsup_{n \to \infty} \frac{1}{\binom{n}{2}} \log \P_{n,p}(
F) \leq - \inf_{f \in F} \Ip(f)\,,
\]
and for any open $U \subseteq \wt \cW_0$,
\[
\liminf_{n \to \infty} \frac{1}{\binom{n}{2}} \log \P_{n,p}(
U) \geq - \inf_{f \in U} \Ip(f)\,.
\]
\end{theorem}

The machinery developed by Chatterjee and Varadhan reduces the problem
determining the large deviation rate function for dense random graphs
to solving a variational problem on graphons.
For any nice graph parameter $\tau \colon \cW_0\to\R$, any $p \in [0,1]$, and any $t \in
\tau(\cW_0)$, let
\begin{equation}
\label{eq:variational}
\phi_\tau (p,t) := \inf\set{\Ip(f) \colon f \in \cW_0~,~ \tau(f) \geq t}\,.
\end{equation}
Since $\Ip$ is lower-semicontinuous on $\wt \cW_0$, the infimum in~\eqref{eq:variational} is always attained.

The following result was stated in~\cite[Thm~4.1 and Prop.~4.2]{CV11} for the graph parameter $\tau = t(K_3,
\cdot)$. We state its generalization to the class of nice graph parameters as per Definition~\ref{def:nice}.

\begin{theorem}[Variational problem]
  \label{thm:variational}
Let $\tau\colon \cW_0\to\R$ be a nice graph parameter and $G_n\sim \cG(n,p)$. Fix $p \in (0,1)$ and $t < \max(\tau)$.
Let $\phi_\tau(p,t)$ denote the solution to~\eqref{eq:variational}.
Then
\begin{align}
\lim_{n \to \infty} \frac{1}{\binom{n}{2}} \log \P\paren{\tau(G_n)
  \geq t} &= - \phi_\tau(p,t) \,.
\end{align}
Let $F^*$ be the set of minimizers for~\eqref{eq:variational} and let $\wt F^*$ be its image in $\wt\cW_0$.
Then $\wt F^*$ is a non-empty compact subset of $\wt \cW_0$. Moreover, for each $\epsilon >
0$ there exists $C = C(\tau, \epsilon, p, t)>0$ so
that for all $n$,
  \[
  \P\cond{\delta_\square (G_n, \wt F^*) < \epsilon}{\tau(G_n)
    \geq t} \geq 1 - e^{-C 
    n^2}\,.
  \]
  In particular, if $\wt F^* = \{ f^*\}$ for some $f^* \in \wt\cW_0$ then
  the conditional distribution of $G_n$ given the event $\tau(G_n) \geq t$
  converges to the point mass at $f^*$ as $n \to \infty$.
\end{theorem}

Observe that by considering $-\tau$ (also a nice graph parameter) one
obtains the same result for lower tail deviations.
%
%
%
The intuition behind
the second part of Theorem~\ref{thm:variational} is
that the probability that $\delta_\square(G_n,\wt F^*) \geq \e$ conditioned on
$\tau(G_n) \geq t$ can be again computed using Theorem~\ref{thm:ldp} and
shown to be exponentially smaller than that of the probability of the
event $\tau(G_n) \geq t$.

The proof of Theorem~\ref{thm:variational} is a straightforward
extension of the arguments of \cite[Thm.~4.1]{CV11} to nice graph parameters. A technical
condition needed to complete the proof of Theorem~\ref{thm:variational} is given by the following lemma,
key to which are the attributes of a nice graph parameter.

\begin{lemma}
  \label{lem:tau-cont}
  Let $\tau\colon\cW_0\to\R$ be a nice graph parameter. For any
  $p \in (0,1)$, the map $t \mapsto \phi_\tau(p,t)$ is continuous
  on $\tau(\cW_0)$.
\end{lemma}

\begin{proof}
Fix $0<p<1$. By definition, $t \mapsto \phi_\tau(p,t)$ is
non-decreasing. To prove left-continuity, consider $a \in \R$. Since $\Ip$ is
lower-semicontinuous on $\cW_0$, the set $\set{f : \Ip(f) \leq a}$ is closed
in $\cW_0$, thus also compact by the compactness of $\cW_0$. Since
$\tau$ is continuous on $\cW_0$, the set $\set{\tau(f) : \Ip(f) \leq a}$ is
compact and in particular closed. Note that the latter is precisely the pre-image of
$(-\infty, a]$ under the inverse of $t \mapsto \phi_\tau(p,t)$. As this pre-image is
closed for any $a \in \R$, left-continuity follows.

In order to prove right-continuity it suffices to show that for every
$t_0 \in \tau(\cW_0)$ with $t_0 < \max(\tau)$ and every $\epsilon >
0$
there exists some $f \in \cW_0$ such that $\Ip(f) <
\phi_\tau(p,t_0) + \e$ and $\tau(f) > t_0$. Indeed,
since $\Ip \colon [0,1] \to \R$ is uniformly continuous (for any
fixed $p$), there exists $\e' > 0$ so that $\abs{\Ip(x) - \Ip(y)} < \e$
whenever $\abs{x-y} < \e'$. Let $f_0$ be the
minimizer of the variational problem~\eqref{eq:variational} for $t =
t_0$. The local extrema condition in Definition~\ref{def:nice} now implies that
there is some $f \in \cW_0$ with $\tau(f) > t_0$ and $\norm{f -
  f_0}_\infty < \e'$. Hence, $\Ip(f) < \Ip(f_0) + \e =
\phi_\tau(p,t_0) + \e$, as desired.
\end{proof}

\section{The phase diagram for subgraph densities and the spectral radius} \label{sec:phase}

\subsection{Subgraph density} \label{sec:phase-subgraph}

In this section we prove Theorem~\ref{thm:count-ldp}, characterizing the
phase diagram of upper tails (replica symmetry vs.\ symmetry breaking) of
the density of a fixed $d$-regular subgraph.

Establishing the replica symmetric phase will hinge on a generalized form
of H\"older's inequality which appeared in~\cite{Fin92}. We include its short proof for
completeness.

\begin{theorem}[Generalized H\"older's inequality]
  \label{thm:gen-holder}
  Let $\mu_1,\ldots,\mu_n$ be probability measures on $\Omega_1,\ldots,\Omega_n$, resp., and
  let $\mu=\prod_{i=1}^n \mu_i$ be the product measure on $\Omega = \prod_{i=1}^n \Omega_i$.
  Let $A_1, \ldots, A_m$ be nonempty subsets of $[n] = \set{1, \dots, n}$ and
  write $\Omega_{A}=\prod_{\ell\in A}\Omega_\ell$ and $\mu_{A} = \prod_{\ell\in A}\mu_\ell$.
  Let $f_i \in L^{p_i}\left(\Omega_{A_i}, \mu_{A_i}\right)$ with $p_i\geq 1$
  for each $i\in[m]$ and suppose in addition that $\sum_{i: \ell\in A_i} (1/p_i) \leq 1$
  for each $ \ell\in [n]$. Then
 \[
  \int \prod_{i=1}^m \left|f_i\right| \ d\mu
  \leq \prod_{i=1}^m \left(\int \left|f_i\right|^{p_i} \ d\mu_{A_i}\right)^{1/p_i}\,.
 \]
 In particular, when $p_i=d$ for every $i\in[m]$ we have $\int \prod_{i=1}^m |f_i| \ d\mu \leq \prod \left(\int |f_i|^d \ d\mu_{A_i}\right)^{1/d}$.
\end{theorem}

\begin{proof}
  The proof carries by induction on $n$ with the trivial base case of $n=0$.
  By Fubini's theorem,
  \begin{align*}
  \int_{\Omega} \prod_{i=1}^m \abs{f_i} \ d\mu
  &= \int_{\Omega} \prod_{i : n \in A_i} \abs{f_i} \prod_{i : n
    \notin A_i} \abs{f_i} \ d\mu
  = \int_{\Omega_{[n-1]}} \bigg(\int_{\Omega_n} \prod_{i : n \in A_i}
    \abs{f_i} \ d \mu_n \bigg) \prod_{i : n
    \notin A_i} \abs{f_i}\ d\mu_{[n-1]}\,.\,
  \end{align*}
  where the argument of each $f_i$ is the restriction of $x\in\Omega$ to the coordinates of $A_i$,
  denoted by $x_{A_i}$.
H\"{o}lder's inequality (along with Jensen's inequality if $\sum_{i:n\in A_i} (1/p_i)$ is less than $1$) implies that
\[ \int_{\Omega_n} \prod_{i : n \in A_i}
    \abs{f_i} \ d \mu_n \leq \prod_{i:n \in A_i} \bigg( \int_{\Omega_n} \abs{f_i}^{p_i} \ d \mu_n\bigg)^{1/p_i}\,,
 \]
 thus for each $i$ with $n\in A_i$ we can let $f_i^*:\Omega_{[n-1]}\to\R$
 denote the averaging map $\left(\int_{\Omega_n} \abs{f_i}^{p_i}\ d\mu_n\right)^{1/p_i}$ and obtain that
  \begin{align*}
    \int_{\Omega} \prod_{i=1}^m \abs{f_i} \ d\mu \leq
     \int_{\Omega_{[n-1]}} \prod_{i : n \in A_i} f_i^* \prod_{i : n
    \notin A_i} \abs{f_i}\ d\mu_{[n-1]}\,.\,
  \end{align*}
 Now, the functions $f_i^*$ correspond to $A_i^* = A_i\setminus\{n\}$ and therefore, thanks to the assumption that
 $\sum_{i: \ell\in A_i} (1/p_i) \leq 1$ for each $\ell\in[n-1]$ we can apply the induction hypothesis and infer that
 \begin{align*}
     \int_{\Omega} \prod_{i=1}^m \abs{f_i} \ d\mu &\leq
          \prod_{i : n \in A_i}\bigg(\int_{\Omega_{[n-1]}} \paren{f_i^*}^{p_i} \ d\mu_{[n-1]} \bigg)^{1/p_i} \prod_{i : n
    \notin A_i} \bigg(\int_{\Omega_{[n-1]}} \abs{f_i}^{p_i}\ d\mu_{[n-1]}\bigg)^{1/p_i} \\
    &= \prod_{i=1}^m \bigg(\int_{\Omega} \abs{f_i}^{p_i} \ d\mu \bigg)^{1/p_i} \,,
 \end{align*}
 as required.
\end{proof}

It is helpful to compare Theorem~\ref{thm:gen-holder} with the standard H\"older's inequality for the case where
$p_i = d$ for all $i$. A direct application of H\"older's inequality produces the inequality
$\norm{\prod f_i}_1 \leq \prod_i \norm{f_i}_m$, whereas Theorem~\ref{thm:gen-holder} exploits the extra
assumption that $\#\{i:\ell\in A_i\}\leq d$ for all $\ell\in[n]$ and gives the stronger inequality
$\norm{\prod f_i}_1 \leq \prod_i \norm{f_i}_d$.
For instance,
\[
\paren{\int f_1(x,y) f_2(y,z) f_3(x,z) \ dxdydz}^2
\leq \prod_i \paren{\int f_i(x,y)^2 \ dxdy}\,.
\]
Placing this in the context of subgraph densities, as an immediate corollary of Theorem~\ref{thm:gen-holder}
(in the special case that $p_i=d$ for all $i$) we have the following inequality.

\begin{corollary}
  \label{cor:reg-holder}
  Let $H$ be a graph whose maximum degree is at most $d$, and let $f \in \cW$. Then
  \[ t(H, f)   \leq \norm{f}_{d}^{e(H)}\,.\]
\end{corollary}

Recall that Theorem~\ref{thm:variational} reduces the problem of finding the
phase boundary to determining whether the constant function $r$ is a
solution for the variational problem of minimizing $\Ip(f)$ over $f
\in \cW_0$ subject to $t(H, f) \geq r^{e(H)}$, where $H$ is some fixed graph.
In light of the above corollary, it is important to estimate $\Ip(f)$
for functions $f\in\cW_0$ with $\norm{f}_d=r$, as addressed next.

\begin{lemma}
  \label{lem:Ld/L1-Ip}
  Let $0 < p < 1$ and let $f\in \cW_0$. Suppose that $d\geq 1$ and $0<r<1$ are such that
  the point $(r^d,\Ip(r))$ lies on the convex minorant of $x \mapsto
  \Ip(x^{1/d})$. If in addition either
  \begin{enumerate}[(a)]
    \item \label{it-Ld/L1-Ip-Ld} $p<r<1$ and $\norm{f}_{d} \geq r$, or
  \item \label{it-Ld/L1-Ip-L1} $0<r<p$ and $\norm{f}_{d} \leq r$,
  \end{enumerate}
     then $\Ip(f) \geq \Ip(r)$, with equality occurring if and only if
  $f\equiv r$.


\end{lemma}

\begin{proof}
  Let $\psi(x) = \Ip(x^{1/d})$ and let $\hat\psi$ be the convex
  minorant of $\psi$. By Jensen's inequality, 
  \[
  \Ip(f) = \int \psi(f(x,y)^d) \ dxdy
         \geq \int \hat\psi(f(x,y)^d) \ dxdy
         \geq \hat\psi\paren{\int f(x,y)^d \ dxdy}
         = \hat\psi\paren{\norm{f}_{d}^d} = \Ip(\norm{f}_d)\,.
  \]
  The fact that $\Ip(x)$ is decreasing along $[0,p]$ and increasing along $[p,1]$ (see \S\ref{app:minorant})
  implies that under either of the assumptions in Part~\eqref{it-Ld/L1-Ip-Ld} and Part~\eqref{it-Ld/L1-Ip-L1}
  we have $\Ip(\norm{f}_d) \geq \Ip(r)$, hence $\Ip(f) \geq \Ip(r)$.
  Since $\hat\psi$ is not linear in any neighborhood of $r^d$, equality
  can occur if and only if $f = r$.
\end{proof}

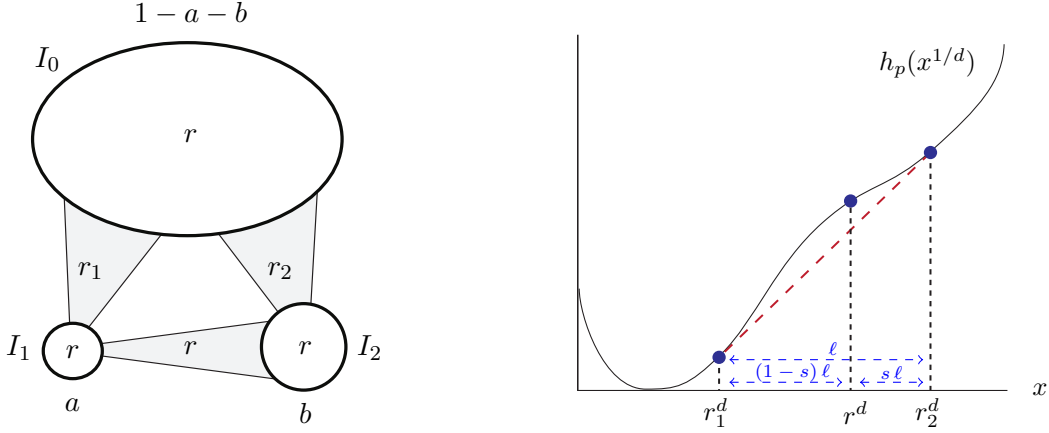
\begin{figure}
  \centering
  \begin{tikzpicture}
    \node (tri) at (0,0)
    {\includegraphics[scale=.6]{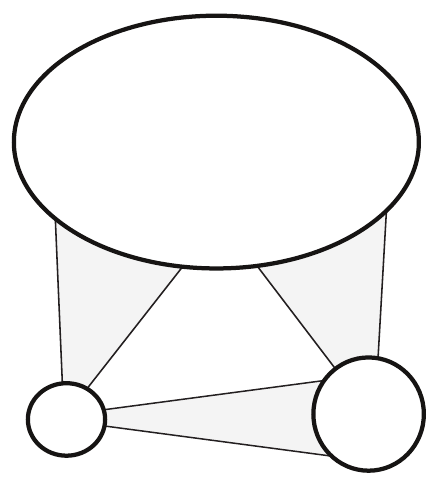}};

    \node (plot) at (8cm,0)
    {\includegraphics{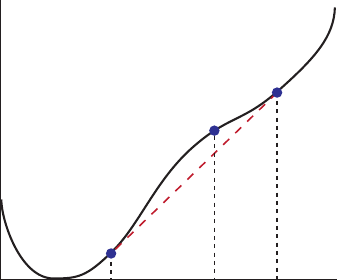}};

    \begin{scope}[shift={(tri.south west)},x={(tri.south east)},y={(tri.north west)}]
      \node at (.17,.14) {$r$};
      \node at (.82,.15) {$r$};
      \node at (.5,.7) {$r$};
      \node at (.22,.35) {$r_1$};
      \node at (.75, .35) {$r_2$};
      \node at (.5, .15) {$r$};
      \node at (.17, 0) {$a$};
      \node at (.82, -.02) {$b$};
      \node at (.5,1.02) {$1-a-b$};
      \node at (.1,.9) {$I_0$};
      \node at (0.02,.15) {$I_1$};
      \node at (1,.15) {$I_2$};
    \end{scope}

    \begin{scope}[shift={(plot.south west)},x={(plot.south
        east)},y={(plot.north west)}, font=\small]
      \node at (1.05,0.03) {$x$};
      \node at (.8,.9) {$\Ip (x^{1/d})$};
      \node at (.33,-.03) {$r_1^d$};
      \node at (.8,-.03) {$r_2^d$};
      \node at (.64,-.03) {$r^d$};
      \draw[<->, color=blue, dashed] (.36,.05) -- (.61,.05) ;
      \draw[<->, color=blue, dashed] (.65,.05) -- (.79,.05) ;
      \node[color=blue,font=\tiny] at (.5,.075) {$(1-s)\,\ell$};
      \node[color=blue,font=\tiny] at (.72,.075) {$s\,\ell$};
      \draw[<->, color=blue, dashed] (.36,.11) -- (.79,.11) ;
      \node[color=blue,font=\tiny] at (.59,.14) {$\ell$};
    \end{scope}
  \end{tikzpicture}
  \caption{The construction in Lemma~\ref{lem:break-count}.}
  \label{fig:construction}
\end{figure}

The final element needed for the proof of Theorem~\ref{thm:count-ldp} is a
construction that outperforms the constant
graphon in the symmetry breaking regime. This is achieved by the following lemma.

\begin{lemma}
  \label{lem:break-count}
  Let $H$ be a $d$-regular graph. Fix $0 < p \leq r < 1$ so that $(r^d, \Ip(r))$ is not
  on the convex minorant of $x \mapsto \Ip(x^{1/d})$.
  Then there exists $f \in \cW_0$ with $t(H, f) > r^{e(H)}$ and
  $\Ip(f) < \Ip(r)$.
\end{lemma}

\begin{proof}
  Since $(r^d, \Ip(r))$ does not lie on the convex minorant
  of $x\mapsto \Ip(x^{1/d})$, there necessarily
  exist $0 \leq r_1 < r < r_2 \leq 1$ such that the point $(r^d,
  \Ip(r))$ lies strictly above the line segment joining $(r_1^d,
  \Ip(r_1))$ and $(r_2^d, \Ip(r_2))$. Letting $s$ be such
  that
  \[ r^d = s r_1^d + (1-s) r_2^d\,,\]
   we therefore have
  \begin{equation}
    \label{eq-s-r1-r2-inequality}
    s \Ip(r_1) + (1-s)
  \Ip(r_2) < \Ip(r)\,.
  \end{equation}
  Let $\e > 0$ and define
  \begin{align*}
   a &= s\e^2\,, \qquad b =   (1-s)\e^2 + \e^3\,,\\
   I_0 &= [a,1-b]\,,\qquad I_1 = [0,a]\,, \qquad I_2 = [1-b,1]\,,
  \end{align*}
  noting that for $a<1-b$ for any sufficiently small $\epsilon$.
  Define $f_\e \in \cW_0$ by
   \[
   f_\epsilon(x,y) = \begin{cases}
     r_1 & \mbox{if $(x,y)\in (I_0 \x I_1) \cup (I_1 \x I_0)$}\,, \\
     r_2 & \mbox{if $(x,y)\in (I_0 \x I_2) \cup (I_2 \x I_0)$}\,, \\
     r   & \mbox{otherwise}\,.
    \end{cases}
   \]
(See Fig.~\ref{fig:construction} for an illustration of this construction.) We claim that
  \begin{equation}
    \label{eq-t(H,f)-r^d}
      t(H, f_\e) - r^{e(H)} = v(H) \paren{ a(r_1^d - r^d) + b(r_2^d - r^d)
  }r^{e(H) - d} + O(\e^4)\,.
  \end{equation}
  Indeed, the only embeddings of $H$ that have values different from
  $r^{e(H)}$ are those where at least one vertex of $H$ is mapped to $I_1\cup I_2$.
  Since $a$ and $b$ are both $O(\e^2)$, in order to compute
  $t(H, f_\e) - r^{e(H)}$ up to an $O(\e^4)$
  error we need only consider embeddings of $H$ where precisely one
  vertex gets mapped to $I_1\cup I_2$. Denote this vertex of $H$ by $u$,
  and observe that if $u$ is mapped to $I_1$ then the contribution
  to $t(H, f_\e) - r^{e(H)}$ is $(r_1^d - r^d) r^{e(H) - d}$ since $H$
  is $d$-regular. Similarly, if $u$ is mapped to $I_2$ then
  the contribution is $(r_2^d - r^d) r^{e(H) - d}$.
  Putting everything together yields~\eqref{eq-t(H,f)-r^d}.

  By definition of $a,b$ and $s$ we have
  \[
  a(r_1^d - r^d) + b(r_2^d - r^d)
  = s\e^2 (r_1^d - r^d) + ((1-s)\e^2 + \e^3) (r_2^d - r^d)
  = \e^3(r_2^d - r^d)\,.
  \]
  Recalling that $r_2 > r$ and plugging the last equation in~\eqref{eq-t(H,f)-r^d} it now follows that $t(H,f_\e) > r^{e(H)}$ for any sufficiently small $\e >0$.
  At the same time, we also
  have
  \begin{align*}
    \Ip(f_\e) - \Ip(r)
    &=
    2a(1-a-b)(\Ip(r_1) - \Ip(r)) + 2b(1-a-b)(\Ip(r_2) - \Ip(r))
    \\
    &= 2(1-a-b)\paren{a \Ip(r_1) + b \Ip(r_2) - (a+b) \Ip(r)}
    \\
    &= 2(1-a-b) \e^2 \paren{s \Ip(r_1) + (1-s)\Ip(r_2) - \Ip(r) + (\Ip(r_2)
      - \Ip(r))\e} \,.
  \end{align*}
  Revisiting~\eqref{eq-s-r1-r2-inequality} we conclude that $\Ip(f_\e) < \Ip(r)$ for any sufficiently small $\e > 0$.
\end{proof}

We now have all the ingredients needed for establishing the phase diagram of upper tail deviations for subgraph densities.

\begin{proof}[\textbf{\emph{Proof of Theorem~\ref{thm:count-ldp}}}]
For Part~\eqref{it-thm1-1}, by applying Theorem~\ref{thm:variational} to the graph parameter $t(H,
  \cdot)$, it suffices to show that the constant function $r$ is the
  unique element $f \in \cW_0$ minimizing $\Ip(f)$ subject to $t(H, F)
  \geq r^{e(H)}$. Indeed, by
  Corollary~\ref{cor:reg-holder}, $t(H, F)
  \geq r^{e(H)}$ implies that $\norm{f}_d \geq r$, and by
  Lemma~\ref{lem:Ld/L1-Ip} Part~\eqref{it-Ld/L1-Ip-Ld}, $\Ip(f) \geq \Ip(r)$ with equality if and
  only if $f$ is the constant function $r$.

To prove Part~\eqref{it-thm1-2}, let $ F^* \subset \wt \cW_0$ be the set of minimizers for the variational
  problem \eqref{eq:variational} with the graph parameter $t(H,
  \cdot)$. Then  $ F^*$ does not contain
  the constant function $r$ by Lemma~\ref{lem:break-count}, nor does it contain any constant function
  of value $r'\neq r$ (when $r'>r$ one has $\Ip(r') > \Ip(r)$, whereas if $r'<r$ then $t(H, f) < r^{e(H)}$).
  Let $\ER\subset \wt \cW_0 $ be the set of constant
  graphons. Since $ F^*$ and $\ER$ are disjoint and both are compact, $\delta_\square( F^*, \ER) > 0$. The desired
  result follows from applying
  Theorem~\ref{thm:variational} with $\e = \delta_\square( F^*, \ER)/2$.

When $d=2$, the phase boundary is explicitly given by Lemma~\ref{lem:d2-boundary},
thus concluding the proof.
\end{proof}


%

%

One can also ask what the phase diagram is for lower tail deviations of
subgraph densities. We next show that for certain bipartite graphs there is replica symmetry everywhere for lower tails.

A beautiful conjecture of Erd\H{o}s and Simonovits~\cite{Sim84} and
Sidorenko~\cite{Sid93} (from here on referred to as \emph{Sidorenko's
  conjecture}) states that every bipartite graph $H$ satisfies $t(H,
G) \geq t(K_2, G)^{e(H)}$ for every graph $G$. The conjecture was
verified for various graphs $H$ (e.g., trees, even cycles~\cite{Sid93},
hypercubes~\cite{Hat10}, bipartite graphs with one vertex complete to the other part~\cite{CFS10}). As it turns out, for any such graph $H$
the lower tail deviations are always replica symmetric (no phase transition).

\begin{proposition}
  \label{prop:sidorenko-lower}
Fix $0 < r \leq p < 1$, let $G_n\sim \cG(n,p)$ be the Erd\H{o}s-R\'enyi random graph and let $H$ be a fixed bipartite graph for which Sidorenko's conjecture holds. Then
  \[
  \lim_{n \to \infty} \frac{1}{\binom{n}{2}} \log \P\paren{t(H, G_n)
    \leq r^{e(H)} } = - \Ip(r)
  \]
  and furthermore, for every $\e > 0$ there exists some constant $C=C(H,\e,p,r)>0$ such that
  \[
  \P \cond{\delta_\square(G_n, r) < \e}{t(H, G_n) \leq
    r^{e(H)}} \geq 1 - e^{- C n^2}\,.
  \]
\end{proposition}
\begin{proof}
  Applying Theorem~\ref{thm:variational} with $-t(H,\cdot)$, it suffices to show that the
  constant function $r$ is the unique element $f \in \cW_0$ minimizing
  $\Ip(f)$ subject to $t(H, f) \leq r^{e(H)}$. Since $H$ satisfies
  Sidorenko's conjecture, $t(H,f) \geq \norm{f}_1^{e(H)}$ for all $f
  \in \cW_0$. Thus, if $t(H, f) \leq r^{e(H)}$ then $\norm{f}_1 \leq r$,
  and so by Lemma~\ref{lem:Ld/L1-Ip} Part~\eqref{it-Ld/L1-Ip-L1} (applied to the case $d=1$, noting that then $\Ip(x^{1/d})$ is itself convex) we have
  $\Ip(f) \geq \Ip(r)$ with equality
  if and only if $f$ is the constant function $r$.
\end{proof}


\subsection{Largest eigenvalue} \label{sec:phase-eigen}

In this section we prove Theorem~\ref{thm:spectral-ldp}, addressing the
phase boundary for large deviations in the spectral norm.
The proof will follow a similar route as the previous section, yet
first we must show that the spectral norm is a nice graph parameter.

For a graph $G$ on $n$ vertices, let $\lambda_1(G)$ be the largest
eigenvalue of its adjacency matrix $A_G$.
Since $A_G$ is symmetric, $\lambda_1(G) = \norm{A_G}_{\op}$, and therefore over $\cW$ we have $\norm{f^G}_\op \geq \lambda_1(G)/n$.
It is easy to verify (see Lemma~\ref{lem:eigen-op} below) that in fact
$\norm{f^G}_\op = \lambda_1(G)/n$, thus the operator norm
on $\cW$ is the graphon extension of the (normalized) largest eigenvalue. Furthermore,
as we show below, $\norm{\cdot}_\op$ is (uniformly) continuous
w.r.t.\ the cut-metric, and the local extrema condition in Definition~\ref{def:nice} is
satisfied as well, thus $\norm{\cdot}_\op$ is a nice graph parameter.

\begin{lemma} \label{lem:eigen-op}
The function $\norm{\cdot}_\op$ is a continuous extension of the normalized graph spectral norm, i.e., $\lambda_1(G)/n$ for a graph $G$ on $n$ vertices, to $(\wt\cW_0,\delta_\square)$. Moreover, $\norm{\cdot}_\op$ is a nice graph parameter.
\end{lemma}

\begin{proof}
We first show that $\norm{f}_\op = \lambda_1(G)/n$ for any graph $G$ on $n$ vertices.
  Clearly, the largest eigenvector of $A_G$, the adjacency matrix of $G$, can be turned
  into a step function $u \colon [0,1] \to \R$ such that $T_{f^G} u =
  (\lambda_1(G)/n) u$, and so $\norm{f^G}_\op \geq
  \lambda_1(G)/n$. Conversely, for any $u \colon [0,1] \to \R$
  we consider the step function $u_n \colon [0,1] \to \R$ such that for
  any $1 \leq i \leq n$, on the interval $(\frac{i-1}{n},
  \frac{i}{n}]$ it is equal to the average of
  $u$ over that interval. Let $v \in \R^n$ be the vector of values
  of $u_n$. Since $f^G$ is constant in every box $(\frac{i-1}{n},
  \frac{i}{n}] \x (\frac{j-1}{n}, \frac{j}{n}]$, we have
  \[
  \norm{T_f
    u}_2 = \norm{T_f u_n}_2 = \norm{A_G v}_2/n^2 \leq \lambda_1(G) \norm{v}_2
  / n^2 = \lambda_1(G) \norm{u_n}_2 / n \leq \lambda_1(G) \norm{u}_2
  /n,
  \]
  where the last inequality is due to convexity. It follows that
  $\norm{f}_\op = \lambda_1(G)/n$.


Next, we will argue that
\begin{equation}
    \label{eq:op-cut}
   \norm{f}_\op^4 \leq 4\norm{f}_\square \qquad
   \mbox{for any symmetric measurable $f \colon [0,1]^2 \to [-1,1]$.}
\end{equation}
Let $u \in L^2([0,1])$ with $\norm{u}_2 = 1$. By Cauchy-Schwarz we can infer that
\begin{align*}
 \norm{T_f u}_2^4
&=\paren{\int \paren{\int f(x,y) u(y) \ dy}^2 \ dx}^2
= \paren{\int f(x,y) f(x,y') u(y) u(y') \ dxdydy'}^2
\\
&\leq
\paren{\int \paren{\int f(x,y)f(x,y') \ dx}^2 \ dydy'} \paren{\int
  u(y)^2 u(y')^2 \ dydy'}\,.
\\
&= \int f(x,y) f(x,y') f(x',y) f(x',y') \ dxdx'dydy'\,.
\end{align*}
For any fixed $x',y'$ we can let
$v_{y'}(x) = f(x,y')$ and $w_{x'}(y) = f(x',y)$, thus rewriting the above as
\[ \int \paren{\int f(x,y) v_{y'}(x) w_{x'}(y) \ dxdy} f(x',y') \ dx'dy' \leq 4 \norm{f}_\square\,,
\]
with the last inequality justified by the fact that for any $g\in\cW$ and $v,w\colon [0,1]\to[-1,1]$
we have $\left| \int g(x,y) v(x) w(y)\ dxdy\right| \leq 4\norm{g}_\square$ by the definition
of the cut-norm, thereby establishing Eq.~\eqref{eq:op-cut}. (The factor of 4 above was due to splitting $v,w$ into positive and negative parts. Indeed,
for $f\colon [0,1]^2\to[0,1]$ the bound in~\eqref{eq:op-cut} remains valid without
the factor of 4 in the right-hand side.)

Consider $f,g \in \wt\cW_0$ and let $\sigma$ vary over all
measure-preserving bijections on $[0,1]$.
We then have
\[
\abs{\norm{f}_\op - \norm{g}_\op}
\leq \inf_\sigma \norm{f - g^\sigma}_\op
\leq \sqrt{2} \inf_\sigma \norm{f - g^\sigma}_\square^{1/4}
= \sqrt{2} \delta_\square(f,g)^{1/4} \, ,
\]
thus implying that $\norm{\cdot}_\op$ is uniformly continuous in $(\wt \cW_0,\delta_\square)$.

Finally, we need to verify the local extrema condition in Definition~\ref{def:nice}. Let $f \in \cW_0$. There are no local non-global minima
since $\norm{(1-\e) f}_\op = (1-\e)\norm{f}_\op < \norm{f}_\op$
unless $\norm{f}_\op= 0$ already. In addition, we claim that $g = \min\set{f + \e, 1}$
satisfies $\norm{g}_\op > \norm{f}_\op$ unless $\norm{f}_\op = 1$
already. Indeed, take some $u \in L^2([0,1])$ which is nonzero (a.e.) and satisfies
 $u \geq 0$ and $T_f u = \norm{f}_\op u$. It suffices to show that $T_{g} u
> T_f u$ on some subset of $[0,1]$ with positive measure.
Let $A = \set{x \in [0,1] : u(x) > 0}$ be the support of $u$, which by our choice
has positive Lebesgue measure $\mu(A)>0$.
 If $\mu(A) < 1$ then $T_f u = u = 0$ (a.e.) on $A^c :=
 [0,1]\setminus A$, so that $f = 0$ on $A^c \x A$. Hence, $g
 = \epsilon$ on $A^c \x A$ and
 $T_g u = \epsilon \norm{u}_1 > 0 = T_f u$ on $A^c$, as desired.
Suppose therefore that $\mu(A) = 1$. If $\norm{f}_\op<1$ we must have $g >
f$ on a subset of positive measure, and hence also $T_{g} u >
T_f u$ on a subset of positive measure, as $u$ has full support. This
shows that $f$ cannot be a local maximum unless $\norm{f}_\op =1$.
\end{proof}

In light of the above lemma, the variational problem under consideration in Theorem~\ref{thm:spectral-ldp} becomes
\begin{equation}
\label{eq:var-op}
\inf\{\Ip(f) : f \in \cW_0, \norm{f}_\op \geq r\} \, .
\end{equation}
We will need the following straightforward inequality relating the operator norm, $\norm{\cdot}_1$ and $\norm{\cdot}_2$.
\begin{lemma}
  \label{lem:L1-op-L2}
  For every $f \in \cW_0$ we have $\norm{f}_{1} \leq \norm{f}_\op \leq \norm{f}_{2}$.
\end{lemma}

\begin{proof}
  For the left inequality, observe that since $f \geq 0$ we have that
  \[
  \norm{f}_{1} = \norm{T_f \mathbf{1}}_{1} \leq \norm{T_f \mathbf{1}}_{2} \leq \norm{f}_\op\,.
  \]
  For the right inequality in the statement of the lemma, let $u \colon [0,1] \to \R$. By Cauchy-Schwarz,
  \[
  \norm{T_f u}_{2}^2 =
  \int \paren{\int f(x,y) u(y) \ dy}^2 \ dx
  \leq \paren{\int f(x,y)^2 \ dydx} \paren{\int u(y)^2 \ dy}
  = \norm{f}_{2}^2 \norm{u}_{2}^2\,,
  \]
  and therefore $\norm{f}_\op \leq \norm{f}_{2}$, as claimed.
\end{proof}

The following lemma is the operator
norm analogue of Lemma~\ref{lem:break-count}, providing a construction that beats the constant
graphon in the symmetry breaking regime.

\begin{lemma}
  \label{lem:break-op}
  Let $0 < p \leq r < 1$ be such that $(r^2, \Ip(r))$ does not lie
  on the convex minorant of $x \mapsto \Ip(\sqrt{x})$.
  Then there exists some $f \in \cW_0$ with $\norm{f}_\op > r$ and
  $\Ip(f) < \Ip(r)$.
\end{lemma}

\begin{proof}
  Let $\e > 0$. Let $f_\e$ be the construction from the proof of Lemma~\ref{lem:break-count} with
  $d=2$, and define the parameters of that construction $r_1,r_2,s,a,b,I_0,I_1,I_2$ as given there.
  Having already demonstrated that $\Ip(f_\e)<\Ip(r)$ for any small enough $\epsilon > 0$,
  it remains to show
  that $\norm{f_\epsilon}_\op > r$.
  To this end, it suffices to exhibit a
  function $u \in L^2([0,1])$ such that $(T_{f_\epsilon}
  u)(x) > r u(x)$ for all $x \in [0,1]$. Let
  \[
  u(x) = \begin{cases} (1-a-b) r_1 & \text{if } x \in I_1\,, \\
                       (1-a-b) r_2 & \text{if } x \in I_2\,, \\
                       r           & \text{if } x \in I_0\,.
        \end{cases}
  \]
  Recall that $f_\epsilon(x,y)$ is $r$ except when $(x,y)\in (I_0\times I_i) \cup (I_i\times I_0)$ where it is $r_i$ for $i=1,2$.
  It now follows that for every $x \in I_1=[0,a]$
  \[
  (T_{f_\e} u)(x)
    = a (1-a-b) r_1 r + b (1-a-b) r_2 r + (1-a-b) r_1 r
    > (1 - a - b) r_1 r = r u(x)\,.
  \]
  Similarly, for any $x \in I_2=[1-b,1]$ we have
  \[
  (T_{f_\e} u)(x)
    >  (1-a-b)  r_2 r
    = r u(x)\,.
  \]
 Finally, if $x \in I_0=[a,1-b]$ then
  \begin{align*}
    (T_{f_\e} u)(x)
    &= a(1-a-b) r_1^2 + b (1-a-b) r_2^2 + (1-a-b) r^2
    = (1-a-b) (r^2 + ar_1^2 + br_2^2)\,.
  \end{align*}
  Plugging in the facts that $r^2=s r_1^2 + (1-s)r_2^2$ while $a=s \epsilon^2$ and $b=(1-s)\epsilon^2+\epsilon^3$, we get that
  \begin{align*}
    (T_{f_\e} u)(x) &
    = (1-\e^2 - \e^3)(r^2 + r^2 \e^2 + r_2^2 \e^3)
    = r^2 + (r_2^2-r^2) \e^3 + O(\e^4)
    > r^2 = r u(x)\,,
  \end{align*}
  where the strict inequality is valid for any sufficiently small $\epsilon>0$ since $r_2 > r$.

  Altogether, $(T_{f_\e} u)(x) > r u(x)$ for all $x\in[0,1]$ and so $\norm{f}_\op > r$, as required.
\end{proof}

\begin{proof}[\emph{\textbf{Proof of Theorem~\ref{thm:spectral-ldp}}}]
To prove Part~\eqref{it-thm2-1}, by
  Theorem~\ref{thm:variational} applied to the graph parameter
  $\norm{\cdot}_\op$ it suffices to show that the constant function
  $r$ is the unique element $f \in \cW_0$ minimizing $\Ip(f)$ subject
  to $\norm{f}_\op \geq r$. Indeed, by Lemma~\ref{lem:L1-op-L2},
  $\norm{f}_2 \geq \norm{f}_\op \geq r$. By Lemma~\ref{lem:Ld/L1-Ip} Part~\eqref{it-Ld/L1-Ip-Ld} and
  Lemma~\ref{lem:d2-boundary} we know that $\Ip(f) \geq \Ip(r)$, with equality if
  and only if $f$ is the constant function $r$.

For Part~\eqref{it-thm2-2}, similar to the proof of Part~\eqref{it-thm1-2} of Theorem~\ref{thm:count-ldp} in~\S\ref{sec:phase-subgraph}, Lemma~\ref{lem:break-op} implies that the set of minimizers
  of the variational problem~\eqref{eq:variational} is disjoint from
  the set of constant graphons. We then apply
  Theorem~\ref{thm:variational} to conclude the proof, with the phase boundary given
  by Lemma~\ref{lem:d2-boundary}.
\end{proof}

The behavior of the lower tails deviations in the spectral norm is similar to
that of the subgraph densities in Proposition~\ref{prop:sidorenko-lower}, where replica symmetry
is exhibited everywhere (no phase transition).

\begin{proposition}
Let $0 < r \leq p < 1$. Let $G_n\sim \cG(n,p)$ be the Erd\H{o}s-R\'enyi random graph and let $\lambda_1(G_n)$ denote the largest
  eigenvalue of its adjacency matrix. Then
\[
\lim_{n \to \infty} \frac{1}{\binom{n}{2}} \log \P\paren{\lambda_1(G_n)
  \leq r}
  = - \Ip(r)
\]
and furthermore, for every $\e > 0$ there is some $C = C(\e,p,r) > 0$
such that for all $n$,
\[
\P\cond{\delta_\square(G_n, r) < \e}{\lambda_1(G_n) \leq nr}
\geq 1- e^{-C n^2} \, .
\]
\end{proposition}

\begin{proof}
 Applying Theorem~\ref{thm:variational} with $\tau =
  -\norm{\cdot}_\op$, it suffices to show that the constant function
  $r$ is the unique element $f \in \cW_0$ minimizing $\Ip(f)$ subject
  to $\norm{f}_\op \leq r$.  By Lemma~\ref{lem:L1-op-L2}, if $f \in \cW_0$
  with $\norm{f}_\op \leq r$ then $\norm{f}_{1} \leq
  \norm{f}_\op \leq r$. It now follows from Lemma~\ref{lem:Ld/L1-Ip} Part~\eqref{it-Ld/L1-Ip-L1}
  (used with $d=1$ bearing in mind that $\Ip(x)$ is convex) that $\Ip(f) \geq
  \Ip(r)$ with equality if and only if $f\equiv r$.
\end{proof}

\section{Exponential random graph models} \label{sec:exp}

Let us review the tools developed by Chatterjee and
Diaconis~\cite{CD} to analyze exponential random
graphs. Define
\[
h(x) := x \log x + (1-x) \log(1-x) \quad\mbox{ for $x \in [0,1]$}\, ,
\]
and for any graphon $f \in \cW_0$ let
\[
h(f) := \int_{[0,1]^2} h(f(x,y)) \ dxdy \, .
\]
The following result from \cite[Thm.~3.1 and Thm.~3.2]{CD} reduces the
analysis of the exponential random graph model in the large $n$ limit
to a variational problem. It was proven with the help of the theory developed
by Chatterjee and Varadhan~\cite{CV11} for large deviations in random graphs.

\begin{theorem}[Chatterjee and Diaconis~\cite{CD}]
  \label{thm:CD-var}
  Let $\tau \colon \wt \cW_0 \to \R$ be a bounded continuous
  function. Let $Z_n = \sum_G \exp\paren{\binom{n}{2} \tau(G)}$ where
  the sum is taken over all $2^{\binom{n}{2}}$ simple graphs $G$ on
  $n$ labeled vertices. Let $\psi_n = \binom{n}{2}^{-1} \log
  Z_n$. Then
  \begin{equation} \label{eq:CD-var}
  \psi := \lim_{n \to \infty} \psi_n = \sup_{f \in \wt \cW_0} (\tau(f) - h(f))\,,
  \end{equation}
  and the set $F^* \subset \wt \cW_0$ of maximizers of this
  variational problem is nonempty and compact.

  Let $G_n$ be a random graph on $n$ vertices drawn from the exponential
  random graph model defined by $\tau$, i.e., with distribution
  $Z_n^{-1} \exp \paren{\binom{n}{2} \tau(\cdot)}$. Then for every
  $\eta > 0$ there exists $C = C(\tau,\eta) > 0$ such that for all $n$,
  \[
  \P\paren{\delta_\square (G_n, F^*) > \eta)} \leq e^{- C n^2}
    \, .
  \]
\end{theorem}

We say that the exponential random graph model has \emph{replica
  symmetry} if the set of maximizers $F^*$ for the variational problem
$\tau(f) - h(f)$ contains only constant functions, and we say that it
has \emph{replica symmetry breaking} if no constant function is a
maximizer. Intuitively, Theorem~\ref{thm:CD-var} implies that when
there is replica symmetry, for large $n$, the random graph behaves
like an Erd\H{o}s-R\'enyi random graph (or a mixture of
Erd\H{o}s-R\'enyi random graphs), while this is not the case when
there is broken symmetry. More precisely we have the following result
(see \cite[Thm.~6.2]{CD}).

\begin{corollary}
  \label{cor:CD-var-break}
  Continuing with Theorem~\ref{thm:CD-var}. Let $\ER \subset \wt \cW_0$
  be the set of constant functions. If $F^* \cap \ER = \emptyset$ then there exist $C, \epsilon > 0$ such that
  for all $n$,
  \[
  \P(\delta_\square(G_n,\ER) > \e) \geq 1 -
  e^{-Cn^2} \, .
  \]
\end{corollary}

To prove Theorem~\ref{thm:exp} using the above tools, we need to analyze
the following variational problem:
\begin{equation}
  \label{eq:exp-var}
  \sup_{f \in \wt\cW_0} (\beta_1 t(K_2, f) + \beta_2 t(H,
    f)^\alpha - h(f)) \, .
\end{equation}
Here is the main result of this section, from which
Theorem~\ref{thm:exp} follows by the results above.

\begin{theorem}
  \label{thm:exp-break}
  Let $H$ be a $d$-regular graph ($d\geq 2$) and fix $\beta_1
  \in \R$ and  $\alpha,\beta_2 > 0$.
  Let $\cE $ denote the corresponding exponential random
  graph model on $n$ labeled vertices as specified in~\eqref{eq-exp-model-H-alpha}.
  \begin{enumerate}[(a)]
  \item\label{item:exp-break-a} If $\alpha \geq d/e(H)$, then $\cE$ has replica symmetry. Moreover, there exists a set $\Gamma\subset \R^2$ of the form
    \[
    \Gamma =
    \{(\beta_1, \varphi(\beta_1)) : \beta_1 < \log(e(H)\alpha - 1) -
    \tfrac{e(H)\alpha}{e(H)\alpha - 1}\} \subset \R^2\quad\mbox{ for some function $\varphi:\R\to\R$}
    \]
   such that when $(\beta_1, \beta_2) \in \R
    \x (0,\infty)\setminus \Gamma$ the set of maximizers of the variational
    problem~\eqref{eq:exp-var} is a single constant function, and when
    $(\beta_1, \beta_2) \in \Gamma$ the set of maximizers  consists of
    exactly two distinct constant functions.
  \item\label{item:exp-break-b} If $0 < \alpha < d/e(H)$ and $\beta_1 \geq \log(d-1) -
    d/(d-1)$ then $\cE$ has replica symmetry. Moreover, the
    variational problem~\eqref{eq:exp-var} is maximized by a unique
    constant function.
  \item\label{item:exp-break-c} If $0 < \alpha < d/e(H)$ and $\beta_1 < \log(d-1) - d/(d-1)$
    then there exists an open interval of values $\beta_2 > 0$ for
    which $\cE$ has broken symmetry, i.e., the set of maximizers of the
    variational problem~\eqref{eq:exp-var} does not contain any constant function.
    Furthermore, this open interval can be taken to
    be $(\ul{\smash{\beta}}_2,\ol\beta_2)$ with
    $\ul{\smash{\beta}}_2 = \ul u^{1-e(H)\alpha} \Ip'(\ul u)/(e(H)\alpha)$ and
    $\ol\beta_2 = \ol u^{1-e(H)\alpha}\Ip'(\ol u)/(e(H)\alpha)$, where $(\ul
    u^d, \Ip(\ul u))$ and $(\ol u^d, \Ip(\ol u))$ are the two points
    where the lower common tangent of $x \mapsto \Ip(x^{1/d})$ touches
    the curve for $p = 1/(1+e^{-\beta_1})$.
  \end{enumerate}
\end{theorem}

When restricted only to constant functions in $\wt \cW_0$, the
variational problem~\eqref{eq:exp-var} becomes the one-dimensional optimization problem
\begin{equation}
  \label{eq:exp-u-op}
  \sup_{0 \leq u \leq 1} (\beta_1 u + \beta_2 u^{e(H)\alpha} -
    h(u)) \, .
\end{equation}
Note that for both \eqref{eq:exp-var} and \eqref{eq:exp-u-op} the
supremum is in fact a maximum due to compactness. Let $u^*$ be the maximizer for
\eqref{eq:exp-u-op}. When there is
replica symmetry, the maximum values attained in \eqref{eq:exp-var}
and \eqref{eq:exp-u-op} are equal, and the exponential random graph
behaves like an Erd\H{o}s-R\'enyi random graph with edge density
$u^*$. It is possible that there are two distinct maximizers $u^*$, in which
case the model behaves like a (possibly trivial) distribution over two
separate Erd\H{o}s-R\'enyi models. In the work of Chatterjee and Diaconis~\cite{CD}, where
the $\alpha = 1$ case was considered, it was shown that $u^*$ as a
function of $(\beta_1, \beta_2)$ experiences a discontinuity
across a curve in the parameter space. Radin
and Yin~\cite{RY} later showed that (when $\alpha = 1$) the limiting free energy density
$\psi$ from~\eqref{eq:CD-var}, as a function in the parameter space $(\beta_1, \beta_2)$, is
analytic except on a first order phase transition curve ending in a
critical point with second order phase transition. (See Fig.~\ref{fig:betaphase} in \S\ref{sec:intro} for a
plot of the location of the discontinuity in the $(\beta_1,
\beta_2)$-phase diagram.)

Here we focus less on the discontinuity of $u^*$ and more on symmetry
breaking. Nevertheless we shall start our analysis by giving a simple geometric interpretation
of the discontinuity of $u^*$.

By definition,
\[
\Ip(x) = h(x) - x \log \frac{p}{1-p} - \log(1-p) \, .
\]
Setting
\[
p = \frac{1}{1 + e^{-\beta_1}} \, ,
\]
so that $\beta_1 = \log \frac{p}{1-p}$, we absorb the linear term in
\eqref{eq:exp-var} and \eqref{eq:exp-u-op} into the entropy term, at
which point these two optimization problems respectively become
\begin{equation}
  \label{eq:exp-var-p}
  \sup_{f \in \wt\cW_0} (\beta_2 t(H, f)^\alpha - \Ip(f) - \log(1-p))
\end{equation}
and
\begin{equation}
\label{eq:exp-u-op-p}
  \sup_{0 \leq u \leq 1} (\beta_2 u^{e(H)\alpha} - \Ip(u) - \log(1-p)) \, .
\end{equation}
By a change of variables $u = x^{1/(e(H)\alpha)}$ in
\eqref{eq:exp-u-op-p} we get the equivalent optimization problem
\begin{equation}
\label{eq:exp-u-op-px}
  \sup_{0 \leq x \leq 1} (\beta_2 x - \Ip(x^{1/(e(H)\alpha)}) - \log(1-p)) \, .
\end{equation}
Observe that $x = x^*$ maximizes \eqref{eq:exp-u-op-px} iff the
tangent to the curve defined by $x \mapsto \Ip(x^{1/(e(H)\alpha)})$ at
$x = x^*$ has slope $\beta_2$ and lies below the curve. Thanks to
Lemma~\ref{lem:hp-shape} from the appendix, we know that $x \mapsto
\Ip(x^{1/\gamma})$ is convex if $0 < \gamma \leq 1$ or if
\begin{equation} \label{eq:exp-pc}
\gamma > 1 \quad\text{and}\quad p \geq
p_0(\gamma) := \frac{\gamma-1}{\gamma-1 + e^{\gamma/(\gamma-1)}} \, .
\end{equation}
Otherwise, $\Ip(x^{1/\gamma})$ has exactly two
inflection points to the right of $x = p^\gamma$, so that the curve starts
convex, becomes concave, and finally turns convex again. In addition,
$\Ip(x^{1/\gamma})$ has an infinite slope at both endpoints. For
any $\beta_2 \in \R$, there is a unique lower tangent of slope
$\beta_2$ to the curve $\Ip(x^{1/(e(H)\alpha)})$, touching the curve
at $x = x^* = (u^*)^{e(H)\alpha}$. As $\beta_2$ varies, $u^*$ increases
continuously with $\beta_2$, except in the situation where the curve of
$\Ip(x^{1/(e(H)\alpha)})$ is not convex and $\beta_2$ is the slope of
the unique lower tangent that touches the curve at two points. In that
case,~\eqref{eq:exp-u-op-p} is optimized at two distinct values of
$u$, denoted by $0<\ul u < \ol u<1$, and as $\beta_2$ increases through this critical point, $u^*$
jumps over the interval $(\ul u, \ol u)$ corresponding to the part of the curve lying above the convex minorant,
then increases continuously afterwards. When $\Ip(x^{1/(e(H)\alpha)})$
is convex, this jump does not occur. See Fig.~\ref{fig:u^*-jump} for an illustration of this process
(the function $\Ip(x^{1/(e(H)\alpha)})$ is plotted not-to-scale in
order to highlight its features). The uniqueness of $u^*$ is stated
below as a lemma.
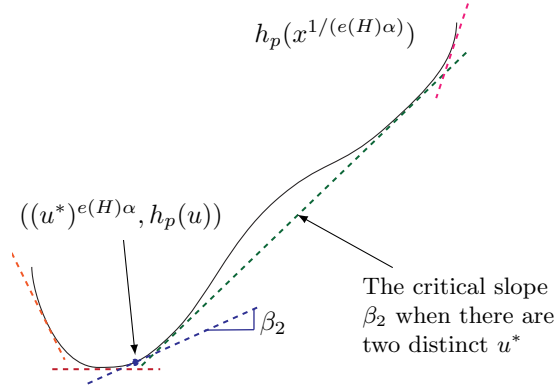
\begin{figure}
\begin{center}
  \begin{tikzpicture}[font=\small]
    \node[anchor=south west] (plot) at (0,0)
    {\includegraphics{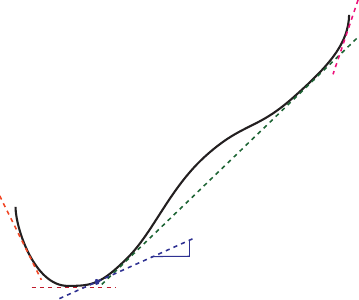}};
    \begin{scope}[x={(plot.south east)},y={(plot.north west)}]
    \node at (.7,.9) {$\Ip(x^{1/(e(H)\alpha)})$};
    \draw[-latex] (.25,.45) node[fill=white] {$((u^*)^{e(H)\alpha}, \Ip(u))$} to (.28,.1);
    \node at (.57,.18) {$\beta_2$};
    \draw (1,.2) node[fill=white,font=\footnotesize,text width=8em]
    {The critical slope $\beta_2$ when there are two distinct $u^*$} edge[-latex] (.62,.45);
    \end{scope}
  \end{tikzpicture}
\end{center}
  \caption{Discontinuity in the symmetric solution $u^*$ due to the geometry of $x\mapsto \Ip(x^{1/(e(H)\alpha)})$.}
\label{fig:u^*-jump}
\end{figure}

\begin{lemma}
  \label{lem:u*-unique}
  If $0 < e(H)\alpha \leq 1$ or if $e(H)\alpha > 1$ and
  $p \geq p_0(e(H)\alpha)$ as defined in \eqref{eq:exp-pc}, then
  the optimization problem \eqref{eq:exp-u-op} is a maximized at a
  unique value of $u$. Otherwise, \eqref{eq:exp-u-op} is maximized at a
  unique $u$ except for a single value of $\beta_2$, where
  the maximum is attained at two distinct $u$'s.
\end{lemma}

We can also represent this jump of $u^*$ in the $(p, u)$-phase diagram
as follows. For each $\gamma > 0$, consider the region
$\cB_{\gamma} \subset [0,1]^2$ containing all points $(p, u)$ such that
$(u^{\gamma}, \Ip(u))$ does not lie on the convex minorant of $x
\mapsto \Ip(x^{1/\gamma})$. When $\gamma < 1$ the region
$\cB_{\gamma}$ is empty, but otherwise it is nonempty. The
geometric argument in the previous paragraph shows that $u$ can never appear as a
maximizer to \eqref{eq:exp-u-op-p} if $(p,u) \in \cB_{e(H)\alpha}$, but
all other values of $u$ can. Thus, for a fixed $p = 1/(1+e^{-\beta_1})$, as
$\beta_2$ increases from $-\infty$ to $\infty$, the point $(p,u^*)$ moves up continuously from $0$ in the
$(p,u)$-phase diagram, and jumps over $\cB_{e(H)\alpha}$ as it
reaches it. Thereafter it resumes moving up until hitting 1. This process is
illustrated on the left of Fig.~\ref{fig:u^*-jump-phase} when $e(H)\alpha = 3$, e.g., when
$H= K_3$ and $\alpha = 1$.
\begin{figure}
\begin{center}
  \begin{tikzpicture}[font=\small]
    \node[anchor=south west] (plot) at (0,0)
    {\includegraphics{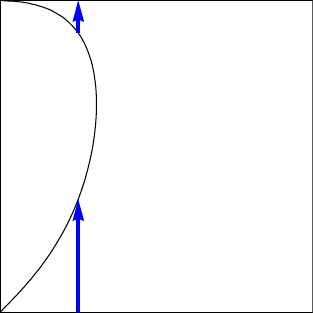}};
    \begin{scope}[x={(plot.south east)},y={(plot.north west)}]
      \node at (.15,.65) {$\cB_3$};
      \node[right,text width=9em] at (.3,.3) {trajectory of $(p,u^*)$ as
        $\beta_2 \nearrow$, $\beta_1$ fixed};
      \node[above] at (0.5,1) {$H = K_3, \quad \alpha = 1$};
      \node[below] at (.5,0) {$p$};
      \node[below=.2em, right] at (0,0) {\footnotesize $0$};
      \node[below=.2em, left] at (1,0) {\footnotesize $1$};
      \node[left] at (0,.5) {$u^*$};
      \node[left=.2em, above] at (0,0) {\footnotesize $0$};
      \node[left=.2em, below] at (0,1) {\footnotesize $1$};
    \end{scope}

    \node[anchor=south west] (plot2) at (8,0)
    {\includegraphics{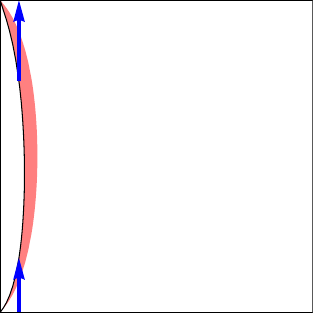}};
    \begin{scope}[shift={(plot2.south west)}, x={(plot2.south east)},y={(plot2.north west)}]
      \node at (.03,.3) {$\cB_{1.8}$};
      \draw node[fill=white,text width = 6em] at (.5,.5) {symmetry broken
        in the shaded region} edge[-latex] (.12,.5);
      \node[above] at (0.5,1) {$H = K_3, \quad \alpha = 0.6$};
      \node[below] at (.5,0) {$p$};
      \node[below=.2em, right] at (0,0) {\footnotesize $0$};
      \node[below=.2em, left] at (1,0) {\footnotesize $1$};
      \node[left] at (0,.5) {$u^*$};
      \node[left=.2em, above] at (0,0) {\footnotesize $0$};
      \node[left=.2em, below] at (0,1) {\footnotesize $1$};
    \end{scope}
  \end{tikzpicture}
\end{center}
  \caption{Discontinuity in the symmetric solution $u^*$ as reflected in the $(p,u$)-phase diagram.}
\label{fig:u^*-jump-phase}
\end{figure}
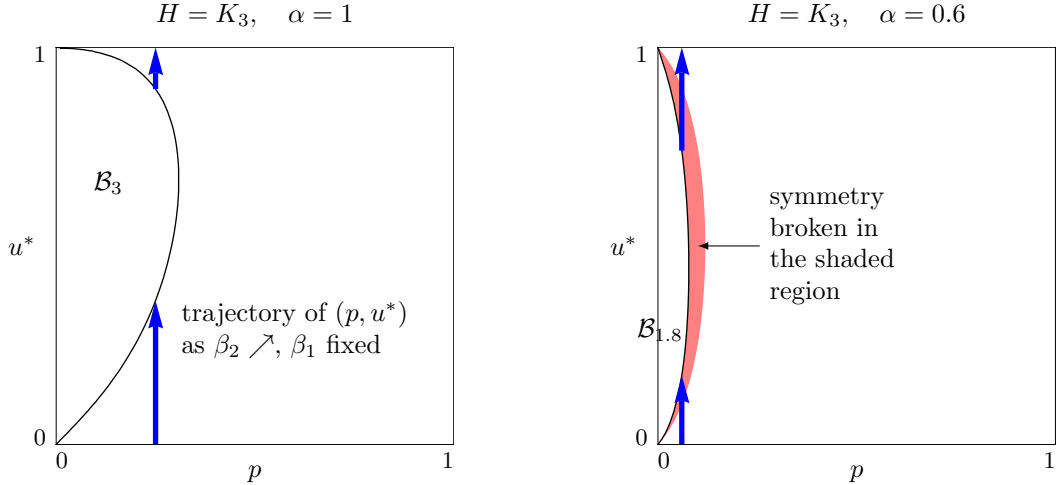

Turning to large deviations, we know from Lemma~\ref{lem:break-count} that there is
broken symmetry for the density of copies of a $d$-regular graph $H$
whenever we are in $\cB_d$. It turns out that the same
is true for the corresponding exponential random graph model given in~\eqref{eq-exp-model-H-alpha}.
As we just saw though, one must remove $\cB_{e(H)\alpha}$ from
the possible solution space for
$(p,u^*)$. Whenever $\gamma < \gamma'$ we have $\cB_{\gamma} \subset \cB_{\gamma'}$
(see Lemma~\ref{lem:hp-d}), and consequently, if $e(H)\alpha \geq d$ then $\cB_{e(H)\alpha}$ covers $\cB_d$ and
so the entire symmetry
breaking phase is removed, leaving replica symmetry
everywhere. This agrees with the results of Chatterjee and Diaconis~\cite{CD} for the
case $\alpha = 1$.
However, when $e(H)\alpha < d$ it is possible to have $(p, u^*)\in\cB_d$, in which case the construction from
Lemma~\ref{lem:break-count} breaks the symmetry. This is shown
on the right of Fig.~\ref{fig:u^*-jump-phase} for the case $d = 2$ and $e(H)\alpha = 1.8$,
e.g., for $H = K_3$ and $\alpha = 0.6$.

\begin{proof}[\emph{\textbf{Proof of Theorem~\ref{thm:exp-break}}}]
The proof of Part~\eqref{item:exp-break-a} is essentially the same as the proof of Theorem~4.1 in
  the work of Chatterjee and Diaconis~\cite{CD}, except now we use
  the generalized H\"older's inequality (Theorem~\ref{thm:gen-holder})
  instead of the usual H\"older's inequality.

  Suppose that $\alpha \geq d/e(H)$. We first need to show that in this case the only maximizers for the variational problem
  \eqref{eq:exp-var-p} are constant functions. Applying Corollary~\ref{cor:reg-holder}, for any $f\in\wt \cW_0$ we have
  \begin{align*}
  \beta_2 t(H, f)^\alpha - \Ip(f)
  &\leq
  \beta_2 \norm{f}_d^{e(H)\alpha} - \Ip(f) \leq
  \beta_2 \norm{f}_{e(H)\alpha}^{e(H)\alpha} - \Ip(f)\,,\end{align*}
  where the last inequality used the assumption on $\alpha$. This in turn is equal to
  \begin{align*}
  & \int (\beta_2f(x,y)^{e(H)\alpha} - \Ip(f(x,y))) \ dxdy
  \leq \sup_{0 \leq u \leq 1} \beta_2 u^{e(H)\alpha} - \Ip(u) \, ,
\end{align*}
showing that the $\beta_2 t(H,f) - \Ip(f)$ is indeed maximized at
constant functions. Furthermore, when $\beta_2 u^{e(H)\alpha} - \Ip(u)$ is
maximized at a unique $u^*$, then equality holds in place of the above inequalities
only for the constant function $f = u^*$. An additional argument is needed
to treat the case when $\beta_2 u^{e(H)\alpha} - \Ip(u)$ is maximized at two
distinct values. Either by checking the equality conditions in the
proof of Theorem~\ref{thm:gen-holder}, or by referring to
\cite{Fin92}, we know that equality in Corollary~\ref{cor:reg-holder}
occurs if and only if $f(x,y) = g(x)g(y)$ for some function $g \colon
[0,1] \to [0, \infty)$. It can then be easily checked that equality in the above sequence of
inequalities can only occur when $f$ is a constant function. The value
of this constant $u^*$ is given by the optimization
problem~\eqref{eq:exp-u-op-p} and its the uniqueness is addressed in
Lemma~\ref{lem:u*-unique}.

We now turn to prove Part~\eqref{item:exp-break-b}. Since $\beta_1 \geq \log(d-1) - d/(d-1)$,
\[
p = \frac{1}{1+ e^{-\beta_1}}
\geq \frac{1}{1 + e^{-\log(d-1) + d/(d-1)}}
= \frac{d-1}{d-1 + e^{d/(d-1)}}
= p_0(d) \, .
\]
By Lemma~\ref{lem:hp-shape}, $x \mapsto \Ip(x^{1/d})$ is convex for this value of $p$. Hence,
$\Ip(f) \geq \Ip(\norm{f}_d)$ by Jensen's inequality with equality if
and only if $f$ is a constant function. Since $t(H, f)
\leq \norm{f}_d$ by Corollary~\ref{cor:reg-holder},
\[
\beta_2 t(H, f)^\alpha - \Ip(f) \leq \beta_2\norm{f}_d^{e(H)\alpha} -
\Ip(\norm{f}_d)
\leq \sup_{0 \leq u \leq 1} (\beta_2 u^{e(H) \alpha} - \Ip(u)) \, ,
\]
with equality iff $f$ is the constant function equal to $u^*$, the
unique maximizer of $\beta_2 u^{e(H) \alpha} - \Ip(u)$. The uniqueness
of $u^*$ follows from Lemma~\ref{lem:u*-unique} together with noting that when
$e(H)\alpha > 1$ we have $p
\geq p_0(d) > p_0(e(H)\alpha)$ as $d > e(H)\alpha$ and $p_0(\cdot)$ is
increasing in $[1,\infty)$.


It remains to prove Part~\eqref{item:exp-break-c}. We have $0 < p < p_0(d)$. Let $0<\ul u< \ol u<1$ be such that the lower common tangent to $x \mapsto
\Ip(x^{1/d})$ touches the curve at $x = \ul u^d$ and $\ol u^d$. Since
$e(H)\alpha < d$, Lemma~\ref{lem:hp-d} implies that the points $(\ul
u^{e(H)\alpha}, \Ip(\ul u))$ and $(\ol u^{e(H)\alpha}, \Ip(\ol u))$
both lie on the convex minorant of $x \mapsto
\Ip(x^{1/(e(H)\alpha)}))$ and do not lie on the common lower
tangent (if there is one). Let $\ul{\smash{\beta}}_2$ and $\ol
\beta_2$ be as in the theorem statement, observing that these are the values of the derivative of $\Ip(x^{1/(e(H)\alpha)})$
at $x = \ul u^{e(H)\alpha}$ and $\ol u^{e(H)\alpha}$, respectively. Then for any $\beta
\in (\ul{\smash{\beta}}_2, \ol \beta_2)$, using the slope
interpretation of $\beta_2$ given in the discussion proceeding this
proof, we see that
the optimization
problem~\eqref{eq:exp-u-op-p} is maximized for some $u^* \in (\ul u,
\ol u)$. At the same time, by Lemma~\ref{lem:break-count} there exists some $f \in \cW_0$
such that $t(H, f) > (u^*)^{e(H)}$ and $\Ip(f) < \Ip(u^*)$. It follows that
\[
  \sup_{f \in \wt\cW_0} (\beta_2 t(H, f)^\alpha - \Ip(f))
  > \beta_2 (u^*)^{e(H)\alpha} - \Ip(u^*)
  = \sup_{0 \leq u \leq 1} (\beta_2 u^{e(H)\alpha} - \Ip(u)) \, ,
\]
and hence $\beta_2 t(H, f)^\alpha - \Ip(f)$ is not maximized at any
constant function.
\end{proof}

\section{Densities of linear hypergraphs in random hypergraphs} \label{sec:hypergraphs}

In this section, we extend our results to densities of linear hypergraphs in random hypergraphs.
Homomorphisms and densities are defined as in graphs. Similarly, for any $r \in [0,1]$ let
\[
\delta_\square(G, r) := \sup_{A_1, \dots, A_k \subset V(G)}
\frac{1}{\abs{V(G)}^k} \big|e_G(A_1, \dots, A_k) - r \abs{A_1} \cdots
  \abs{A_k}\big|
\]
where $e_G(A_1, \dots, A_k)$ is the number of (ordered) hyperedges of
the form $(a_1, \dots, a_k) \in A_1 \x \cdots \x A_k$.
The main result in this section is the following:

\begin{theorem}
  \label{thm:linear-hypergraph-ldp}
  Fix $d, k \geq 2$ and $0 < p \leq r < 1$. Let
  $H$ be a $d$-regular $k$-uniform linear hypergraph. Let
  $G_n\sim \cG^{(k)}(n,p)$ be the random $k$-uniform hypergraph on $n$ vertices with
  hyperedge probability $p$.
  \begin{enumerate}[(i)]
  \item\label{it-thm-hyper-1} If the point $(r^d,\Ip(r))$
  lies on the convex minorant of the function $x \mapsto
  \Ip(x^{1/d})$ then
  \[
  \lim_{n \to \infty} \frac{1}{\binom{n}{k}} \log \P\paren{t(H, G_n)
    \geq r^{e(H)} } = - \Ip(r)
  \]
  and furthermore, for every $\epsilon > 0$ there exists some constant $C=C(H,\e, p, r)>0$ such that
  \[
  \P \cond{\delta_\square(G_n, r) < \e}{t(H, G_n) \geq
    r^{e(H)}} \geq 1 - e^{-C n^k}\,.
  \]

\item\label{it-thm-hyper-2} If the point $(r^d,\Ip(r))$ does not lie on the convex
  minorant of the function $x \mapsto \Ip(x^{1/d})$ then
  \[
  \lim_{n \to \infty} \frac{1}{\binom{n}{k}} \log \P\paren{t(H, G_n)
    \geq r^{e(H)} } > - \Ip(r)
  \]
    and furthermore, there exist $\e, C > 0$ such that
  \[
  \P \cond{
  \inf\big\{  \delta_\square(G_n,s) : 0\leq s \leq 1\big\} > \e}{t(H, G_n) \geq
    r^{e(H)}} \geq 1 - e^{-C n^k} \, .
  \]

  \end{enumerate}
 In particular, when $d = 2$, case \eqref{it-thm-hyper-2} occurs if and only if $p < \left[1 + (r^{-1} - 1)^{1/(1-2r)}\right]^{-1}$.
\end{theorem}

Let $k \geq 2$ be an integer and let $\cW^{(k)}$ be the space of all
bounded measurable functions $[0,1]^k \to \R$ that are symmetric
(i.e., $f(x_1, x_2, \dots, x_k) = f(x_{\pi(1)},
x_{\pi(2)}, \dots, x_{\pi(k)})$ for any permutation $\pi$ of $[k] =
\set{1, 2, \dots, k}$). Let $\cW^{(k)}_0$ denote all symmetric
measurable functions $[0,1]^k \to [0,1]$.
Every $k$-uniform hypergraph $G$ corresponds to a point $f^G \in
\cW_0^{(k)}$ similar to the case for graphs. As before, we can endow $\cW$
with usual $L^p$-norm and in addition have the following cut norm:
\[
\norm{f}_\square := \sup_{S_1, \dots, S_k \subset [0,1]} \int_{S_1 \x
  \cdots \x S_k} f(x_1, \dots, x_k) \ dx_1 \cdots dx_k\,.
\]
This gives rise to the cut distance: for any $f,g \in \cW_0^{(k)}$,
\[
\delta_\square(f,g) := \inf_{\sigma} \norm{f - g^\sigma}_\square
\]
where $\sigma$ ranges over all measure-preserving bijections on
$[0,1]$, and $g^\sigma \in \cW^{(k)}_0$ is defined by $g^\sigma(x_1,
\dots, x_k) = g(\sigma(x_1), \dots, \sigma(x_k))$. Let $\wt
\cW^{(k)}_0$ be the metric space formed by taking equivalences of
points in $\cW^{(k)}_0$ with zero cut-distance.

The space $\cW_0^{(k)}$ is a straightforward generalization of the
space $\cW_0$ of graphons. Unfortunately, it does not fully capture the
richness of the structure of hypergraphs. This notion is closely
related to some initial attempts at generalizing Szemer\'edi's
regularity lemma to hypergraphs (e.g., \cite{Chu91}). The main issue
is that while the regularity lemma generalizes easily to this setting,
there is no corresponding counting lemma for embedding a fixed
hypergraph $H$ unless $H$ is linear (recall that a hypergraph is
linear if every pair of vertices is contained in at most one
hyperedge). The difficulty in extending the results to general $H$ is
related to the intricacies of hypergraph regularity (see, e.g.,
Gowers~\cite{Gow07} and Nagle, R\"odl,
Schacht, and Skokan~\cite{NRS06,RS04}, as well as the recent progress
in this direction by Elek and Szegedy~\cite{ES07,ES08}). Here
we restrict ourselves to the basic setting above which suffices for
controlling densities of linear hypergraphs.

For any $f \in \cW^{(k)}$ and
any $k$-uniform hypergraph $H$, write $V(H) = [m]$ and define
\[
t(H, f) = \int_{[0,1]^k} \prod_{\set{i_1, \dots, i_k} \in E(H)}
f(x_{i_1}, \dots, x_{i_k}) \ dx_1 \cdots dx_m\,.
\]
The Chatterjee-Varadhan theory can be generalized to derive rate
functions for large deviations of $H$-counts, where $H$ is a fixed
linear hypergraph. We outline the modifications and omit the
complete details, as the changes required in the original proofs are mostly straightforward.

We start with a statement generalizing the weak regularity lemma of
Frieze and Kannan~\cite{FK99}. The analytic form of this statement for graphs can be
found in Lov\'asz and Szegedy~\cite[Lem.~3.1]{LS07}.

\begin{theorem}
  \label{thm:hyp-FK}
  For every $\e > 0$ there exists some $M(\e)>0$ such that for every $f \in
  \cW_0^{(k)}$ there exist some $m \leq M(\e)$ and some $g \in
  \cW_0^{(k)}$ with $\delta_\square(f,g) \leq \e$, and such that $g$
  is constant in each box $(\frac{i_1 - 1}m, \frac{i_1}m] \x \cdots \x ( \frac{i_k -
  1}m, \frac{i_k}m]$.
\end{theorem}

Using Theorem~\ref{thm:hyp-FK}, the proof of Lov\'asz and
Szegedy~\cite[Thm.~5.1]{LS07} can be modified to give the following
topological interpretation of this result.

\begin{theorem}
  \label{thm:hyp-LZ}
  For any integer $k \geq 2$, the metric space $(\wt \cW_0^{(k)},
  \delta_\square)$ is compact.
\end{theorem}

Theorems~\ref{thm:hyp-FK} and \ref{thm:hyp-LZ} allow us to generalize
the framework of Chatterjee and Varadhan to $(\wt \cW^{(k)}_0, \delta_\square)$.  The
random hypergraph graph $\cG^{(k)}(n,p)$ corresponds to a random point
$f^{\cG^{(k)}(n,p)} \in \wt \cW^{(k)}$, and therefore it induces a
probability distribution $\P_{n,p}$ on $\wt \cW^{(k)}$ supported on a
finite set of points corresponding to hypergraphs on $n$
vertices.

\begin{theorem}
  \label{thm:hyp-ldp} For each fixed $p \in (0,1)$, the sequence
$\P_{n,p}$ obeys a large deviation principle in the space $(\wt \cW_0^{(k)},
\delta_\square)$ with rate function $\Ip$. Explicitly, for any closed
set $F \subseteq \wt \cW_0^{(k)}$,
\[ \limsup_{n \to \infty} \frac{1}{\binom{n}{k}} \log \P_{n,p}( F)
\leq - \inf_{f \in F} \Ip(f)\,,
\] and for any open set $U \subseteq \wt \cW_0^{(k)}$,
\[ \liminf_{n \to \infty} \frac{1}{\binom{n}{k}} \log \P_{n,p}( U)
\geq - \inf_{f \in U} \Ip(f)\,.
\]
\end{theorem}

To derive large deviation results for subgraph densities in random
graphs, it was crucial that the subgraph densities $t(H, \cdot)$
behaved continuously with respect to the cut topology. The next result
implies that the same is
true when $H$ is a linear hypergraph. The proof is a
straightforward generalization of the proof for graphs (see
\cite[Thm.~3.7]{BCLSV08}).

\begin{theorem}
  Let $H$ be a $k$-uniform linear hypergraph. Then for any $f, g \in \cW_0^{(k)}$,
  \[
  \abs{t(H, f) - t(H, g)} \leq e(H) \delta_\square(f,g)\,.
  \]
\end{theorem}

The rate function for large deviations in $H$-counts is then
determined by the following variational problem.

\begin{theorem}
  \label{thm:hyp-variational}
  Let $H$ be a $k$-uniform linear hypergraph and let $G_n\sim \cG^{(k)}(n,p)$ be the
  random $k$-uniform hypergraph on $n$ vertices with hyperedge probability $p$. For any fixed $p,r \in (0,1)$,
  \begin{equation}\label{eq:hyp-variational}
    \lim_{n \to \infty} \frac{1}{\binom{n}{k}} \log \P\paren{t(H, G_n)
      \geq r^{e(H)}} =  - \inf\set{\Ip(f) \colon f \in \cW_0^{(k)}, t(H, f) \geq
      r^{e(H)}} \, .
\end{equation}
Let $F^*$ be the set of minimizers for~\eqref{eq:hyp-variational} and let $\wt F^*$ be its image in $\wt\cW_0^{(k)}$.
Then $\wt F^*$ is a non-empty compact set. Moreover, for each $\epsilon >
0$ there exists some $C(H,\epsilon, p, r)>0$ so
that for any $n$
  \[
  \P\cond{\delta_\square (G_n, \wt F^*) \geq \epsilon}{ t(H,G_n) \geq r^{e(H)}} \leq e^{-C  n^k}\,.
  \]
  In particular, if $\wt F^* = \{ f^*\}$ for some $f^* \in \wt\cW_0^{(k)}$ then
  the conditional distribution of $G_n$ given the event $t(H,G_n)\geq r^{e(H)}$
  converges to the point mass at $f^*$ as $n \to \infty$.
\end{theorem}

We now turn to study the variational problem~\eqref{eq:hyp-variational} towards the proof of Theorem~\ref{thm:linear-hypergraph-ldp}.
The following inequality is an immediate consequence of
Theorem~\ref{thm:gen-holder}.

\begin{lemma}
  \label{lem:hyp-reg-holder}
  Let $H$ be a $k$-uniform hypergraph with maximum degree at most
  $d$, and let $f \in \cW_0^{(k)}$. Then $t(H, f) \leq \norm{f}_d^{e(H)}$.
\end{lemma}

The following lemma mirrors Lemma~\ref{lem:break-count} for proving
the symmetry breaking phase.

\begin{lemma}
  \label{lem:hyp-break}
  Let $H$ be linear $d$-regular $k$-uniform hypergraph. Let $0 < p
  \leq r < 1$ be such that $(r^d, \Ip(r))$ does not lie on the convex
  minorant of $x \mapsto \Ip(x^{1/d})$. Then there exists $f \in
  \cW_0^{(k)}$ with $t(H, f) > r^{e(H)}$ and $\Ip(f) < \Ip(r)$.
\end{lemma}

The proof of Lemma~\ref{lem:hyp-break} is essentially the same as the
proof of Lemma~\ref{lem:break-count}, with the following modification.
One needs to adjust $f_\e$ into
$f_\e = r + (r_1 - r) \one_A + (r_2 - r) \one_B$, where $A \subset [0,1]^k$
is the union of the box $[0,a] \x [a,1-b]^{k-1}$ along with the $k-1$ other
boxes formed by permuting the coordinates. Similarly $B$ is the union
of $[1-b,1] \x [a,1-b]^{k-1}$ and its coordinate permutations.

\begin{proof}[\textbf{\emph{Proof of Theorem~\ref{thm:linear-hypergraph-ldp}}}]
Our starting point is an application of Theorem~\ref{thm:hyp-variational}.
Suppose $f \in \cW^{(k)}_0$
  satisfies $t(H, f) \geq r^{e(H)}$. By
  Lemma~\ref{lem:hyp-reg-holder}, $\norm{f}_d \geq r$. For Part~\eqref{it-thm-hyper-1} of
  the theorem, Lemma~\ref{lem:Ld/L1-Ip} Part~\eqref{it-Ld/L1-Ip-Ld} (which is also valid for
  $\cW_0^{(k)}$) implies that $\Ip(f) \geq \Ip(r)$ with equality if and
  only if $f$ is the constant function $r$, so the variational
  problem on the right-hand side of \eqref{eq:hyp-variational} has the
  constant function $r$ as the unique minimizer. For Part~\eqref{it-thm-hyper-2} of the
  theorem, Lemma~\ref{lem:hyp-break} implies that the constant
  function $r$ is not in the set of minimizers of the variational
  problem~\eqref{eq:hyp-variational}, and the conclusion follows
  analogously to the proof of Part~\eqref{it-thm1-2} of Theorem~\ref{thm:count-ldp}.
\end{proof}

\section{Graph homomorphism inequalities} \label{sec:hom}

Our goal in this section is to present a short new proof of Theorem~\ref{thm:GT},
stating that for any $d$-regular bipartite graph $G$ and any graph $H$ allowing loops one has
\begin{equation} \label{eq:GT}
\hom(G, H) \leq \hom(K_{d,d}, H)^{|V(G)|/(2d)}\,.
\end{equation}
(This inequality is tight when $G$ is a disjoint union of copies of the complete bipartite graph $K_{d,d}$.)
Generalized to graphons, the inequality states that for any $d$-regular bipartite graph $G$ and
  $f \in \cW_0$
\begin{equation} \label{eq:GT'}
  t(G, f) \leq t(K_{d,d}, f)^{\abs{V(G)}/(2d)}
\end{equation}
(the more general formulation follows from~\eqref{eq:GT} via a standard limiting argument, e.g., see~\cite{GT04}).
As mentioned in the introduction, all previously known proofs of this inequality involved entropy.
In contrast, the following proof does not rely on entropy or limiting arguments and instead is
an immediate consequence of the generalized H\"older's inequality
(Theorem~\ref{thm:gen-holder}).
\begin{proof}[\textbf{\emph{Proof of Theorem~\ref{thm:GT}}}]
Label the vertices on the left-bipartition of $G$ by $[n] = \set{1,
  \dots, n}$, and let $A_i \subset [n]$ be the neighborhood of the $i$-th
vertex on the right-bipartition of $G$. Define
\[
g(x_1, \dots, x_d) := \int f(x_1, y) f(x_2, y) \cdots f(x_d, y) \ dy \qquad\mbox{for any
$x_1, \dots, x_d \in [0,1]$}
\, ,
\]
and write $g(x_A)$ for $x\in [0,1]^n$ and $A=\{i_1,\ldots,i_d\}$ to denote $g(x_{i_1},\ldots,x_{i_d})$.
With this notation,
\[
t(G, f) = \int \prod_{j=1}^n \prod_{i \in A_j} f(x_i, y_j) \ dx_1
\cdots dx_n dy_1 \cdots dy_n
 = \int_{[0,1]^n} g(x_{A_1}) \cdots g(x_{A_n}) \ d x
\leq \norm{g}_d^n
\]
by the generalized H\"older's inequality (Theorem~\ref{thm:gen-holder}).
The result follows from noting that
\begin{equation*}
\norm{g}_d^d = \int \prod_{i=1}^d \prod_{j=1}^d f(x_i, y_j) \ dx_1
\cdots dx_d dy_1 \cdots dy_d = t(K_{d,d}, f)\,.\qedhere
\end{equation*}
\end{proof}

\begin{remark}
A natural question to ask is whether Theorem~\ref{thm:GT} can be extended
to all $d$-regular graphs $G$, as in the case for independent
sets~\cite{Zhao10}. Unfortunately, the answer to this question is
negative. A simple counter-example is $G = K_3$ and $f$ being the graphon
corresponding to the $2 \x 2$ identity matrix. The second author~\cite{Zhao11}
extended Theorem~\ref{thm:GT} to non-bipartite $G$ for certain
families of $f \in \cW_0$, e.g., $\set{0,1}$-valued graphons that are
non-decreasing in both coordinates. Galvin~\cite{Gal} conjectured that
if $G$ is a simple $d$-regular graph and $H$ is a graph allowing
loops then $\hom(G,H)$ is at most the maximum of $\hom(K_{d,d},
H)^{\abs{V(H)}/(2d)}$ and $\hom(K_{d+1}, H)^{\abs{V(H)}/(d+1)}$.
\end{remark}

\section{Open problems} \label{sec:end}

It is natural to ask for extensions of the Chatterjee-Varadhan~\cite{CV11}
large deviations theory to sparse Erd\H{o}s-R\'enyi random graphs (i.e., $\cG(n,p)$ where $p(n)\to 0$ as $n\to\infty$) or to
densities of general (not necessarily linear) hypergraphs in random hypergraphs.
 These may require extensions of
Szemer\'edi's regularity lemma to sparse graphs (see
\cite{Koh97,GS05,CFZ}) and hypergraphs \cite{Gow07,NRS06,RS04}.
Even for $\cG(n,p)$ with fixed $p$, various problems remain open, several of which we highlight below.

\subsubsection*{Minimizers of the variational problem}
It was pointed out by Chatterjee and Varadhan~\cite{CV11}
that no solutions of the variational problem in Theorem~\ref{thm:variational}
are known anywhere in the symmetry breaking phase. This remains the case.
In fact, there is not a single point $(p,r)$ in the symmetry breaking phase
where we can even compute the large deviation rate.
It would be interesting to see whether the minimizers are always 2-step graphons, i.e., graphons
which are constant on each of $[0,w]^2, \big([0,w] \x (w,1]\big) \cup \big((w,1] \x
[0,w]\big)$ and $[w,1]^2$ for some $w \in (0,1)$.

\subsubsection*{Phase boundary for non-regular graphs}
In this paper, we identified the replica symmetric phase for
upper tail deviations in the densities of $d$-regular graphs. The phase boundary for non-regular graphs
remains unknown. We suspect
that a modification of the construction in Lemma~\ref{lem:break-count}
can be used to establish the symmetry breaking phase for some (perhaps all) subgraph density deviations. However,
at present we do not have matching boundaries for the replica symmetric phase.

\subsubsection*{Lower-tail phase transition}
Proposition~\ref{prop:sidorenko-lower} shows that if a bipartite graph $H$
satisfies Sidorenko's conjecture then there is replica symmetry
everywhere in the lower tail deviation of $H$-densities. It is an open
question whether all bipartite graphs satisfy Sidorenko's conjecture,
although it is possible that Proposition~\ref{prop:sidorenko-lower} can be
proved without the full resolution of Sidorenko's conjecture.

When $H$ is not bipartite, the following argument shows that
there exists symmetry breaking in the lower tail, at least for certain
values of $(p,r)$. Let $f$ be the graphon taking the value $0$ on
$[0,\frac12]^2 \cup [\frac12,1]^2$ and the value $p$ elsewhere, so that $t(H,f) = 0$
and $\Ip(f) = \Ip(0)/2$. Let $r_0 \in (0,p)$ be such that $\Ip(0) = 2
\Ip(r_0)$. Then for any $r \in (0, r_0)$ we have $\Ip(f) < \Ip(r)$,
and so $(p,r)$ is in the symmetry breaking phase, resulting in a nontrivial
phase diagram. We currently do not
know the complete lower tail phase diagram for any non-bipartite $H$.

\subsubsection*{Symmetry breaking in exponential random graphs}
Fig.~\ref{fig:betaphase} showed several $(\beta_1, \beta_2)$-phase
plots for the symmetry breaking region given in Part~\eqref{item:exp-break-c} of Theorem~\ref{thm:exp-break}.
However, unlike the situation for large deviations, we do not
know if that is the full region of symmetry breaking. It would be interesting to
  characterize the full set of triples $(\alpha, \beta_1, \beta_2)$ for
  which there is replica symmetry in Theorem~\ref{thm:exp}.

\bigskip
\appendix

\section{The convex minorant of $\Ip(x^{1/\gamma})$} \label{app:minorant}

This appendix contains some technical lemmas about the convex minorant
of $\Ip(x^{1/\gamma})$, which appears throughout the paper.
Let us first informally summarize the claims. It is well known that
\[
\Ip(x) = x
\log \frac{x}{p} + (1-x) \log \frac{1-x}{1-p}
\]
is
a convex function of $x$. When $\gamma > 1$ (here we allow any real $\gamma$,
not just integers), it turns out that there is some $p_0(\gamma)$ for which
$x \mapsto \Ip(x^{1/\gamma})$ is still a convex function when $p \geq p_0$. However, when $p
< p_0$, the function $x \mapsto \Ip(x^{1/\gamma})$ is no longer convex: it
has exactly two inflection points, both to the right of its minimum at
$x = p^\gamma$. The function is concave in the corresponding middle region, whereas
it is convex in the two outer regions.

In the case when $x \mapsto \Ip(x^{1/\gamma})$ is not convex, it has a
unique lower tangent, touching the plot of the function at the points
$(\uq^\gamma, \Ip(\uq))$ and $(\oq^\gamma, \Ip(\oq))$. The convex minorant of $x
\mapsto \Ip(x^{1/\gamma})$ is formed by replacing the middle segment $x \in
(\uq^\gamma, \oq^\gamma)$ by the lower common tangent, as shown
in Fig.~\ref{fig:convex-minorant-gamma}. (The various not-to-scale plots of $\Ip(x^{1/\gamma})$
are shown for illustrative purposes in order to highlight the features
of the plots. In contrast, all the phase diagrams plots are drawn to
scale.)

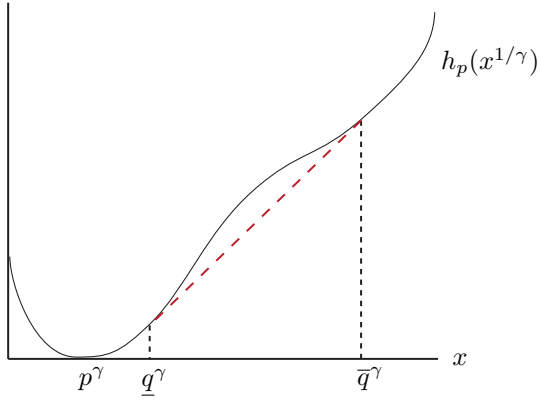
\begin{figure}
\begin{center}
  \begin{tikzpicture}[font=\small]
    \node[anchor=south west] (plot) at (0,0)
    {\includegraphics{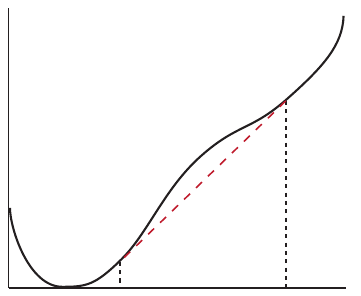}};
    \begin{scope}[x={(plot.south east)},y={(plot.north west)}]
      \path[use as bounding box] (0,-.05) rectangle (1.16,1);
      \node[right] at (.95,0.05) {$x$};
      \node at (1.05,.8) {$\Ip(x^{1/\gamma})$};
      \node at (.22,0) {$p^\gamma$};
      \node at (.35,-.01) {$\uq^\gamma$};
      \node at (.8,0) {$\oq^\gamma$};
    \end{scope}
  \end{tikzpicture}
\end{center}
\caption{Illustration of the convex minorant of $x \mapsto\Ip(x^{1\gamma})$.}
\label{fig:convex-minorant-gamma}
\end{figure}

Let $\cB_\gamma \subset [0,1]^2$ denote the set of all points $(p, q)$ such
that $(q^\gamma, h(q))$ does not lie on the convex minorant of $x \mapsto
\Ip(x^{1/\gamma})$. Each vertical section of $\cB_\gamma$ is thus the interval
$(\uq^\gamma, \oq^\gamma)$ described in the previous paragraph. For example,
$\cB_2$ is the shaded region in Fig.~\ref{fig:B2-phase}, and
the boundaries of $\cB_\gamma$ for additional values of $\gamma$ were plotted in
Figure~\ref{fig:phase-d} (see Section~\ref{sec:intro}).

We shall show that each $\cB_\gamma$ resembles a rotated V-shape. More
importantly, we show that $\cB_\gamma$ strictly contains $\cB_{\gamma'}$ if $\gamma > \gamma'$.

The first lemma describes the shape of the function $x \mapsto \Ip(x^{1/\gamma})$.

\begin{lemma}
  \label{lem:hp-shape}
  The function $x \mapsto \Ip(x^{1/\gamma})$ with domain $(0,1)$ is convex
  for $0 < \gamma \leq 1$. If $\gamma > 1$, and
  \[
  p \geq p_0(\gamma) := \frac{\gamma-1}{\gamma-1 +
    e^{\gamma/(\gamma-1)}} \, ,
  \]
  then the function is also convex. If $\gamma > 1$ and $0
  < p < p_0$, then the function has exactly two inflection points
  (both to the right of $x = p^\gamma$),
  with a region of concavity in the middle. Finally the function has
  infinite derivatives at both endpoints of $(0,1)$.
\end{lemma}

\begin{proof}
When $0 < \gamma \leq 1$, $x \mapsto x^{1/\gamma}$ is convex. Since the composition
of two convex functions is convex we deduce that $x \mapsto \Ip(x^{1/\gamma})$ is convex.

Next, assume $d \geq 1$. We have
\[
\Ip'(x) = \log \frac{x}{1-x} - \log\frac{p}{1-p} \, ,
\qquad
\text{and}
\qquad
\Ip''(x) = \frac{1}{x(1-x)} \, .
\]
Therefore,
\begin{align*}
  \frac{d}{dx} \Ip(x^{1/\gamma}) &= \frac{1}{\gamma} x^{1/\gamma-1} h'_p(x^{1/\gamma})
  \end{align*}
and
  \begin{align*}
  \frac{d^2}{dx^2} \Ip(x^{1/\gamma}) &=
  \frac{1}{\gamma^2} x^{2/\gamma-2} h''(x^{1/\gamma}) +
  \frac{1}{\gamma}\paren{\frac{1}{\gamma}-1} x^{1/\gamma - 2} h'_p(x^{1/\gamma}) \\
  &= \frac{x^{1/\gamma-2}}{\gamma^2} \paren{x^{1/\gamma} \Ip''(x^{1/\gamma}) - (\gamma-1)
    h'_p(x^{1/\gamma})} \, .
\end{align*}
The claim on infinite derivatives easily follows from
the above formulas.
Setting $x = q^\gamma$, we have
\begin{align}
  \frac{d^2}{dx^2} \Ip(x^{1/\gamma}) \Big\vert_{x= q^\gamma}
  &= \frac{q^{1-2\gamma}}{\gamma^2} \paren{q \Ip''(q) - (\gamma-1)\Ip'(q))} \nonumber
  \\
  &= \frac{q^{1-2\gamma}}{\gamma^2}\paren{\frac{1}{1-q} - (\gamma-1)
      \log\frac{q}{1-q}  + (\gamma-1) \log \frac{p}{1-p}} \, . \label{eq:h_p''}
\end{align}
Hence, $\Ip(x^{1/\gamma})$ is convex at $x = q^\gamma$ whenever $\frac{1}{1-q} -
(\gamma-1) \log\frac{q}{1-q} \geq - (\gamma-1) \log \frac{p}{1-p}$ and concave
otherwise. The fact that
\[
\frac{d}{dq}\paren{\frac{1}{1-q} - (\gamma-1)\log\frac{q}{1-q}} =
\frac{\gamma q - \gamma + 1}{q(1-q)^2}
\]
implies that $\frac{1}{1-q} - (\gamma-1)\log\frac{q}{1-q}$ is decreasing until $q =
(\gamma-1)/\gamma$ and then increasing afterwards. It diverges to $+\infty$
at both endpoints of $(0,1)$ and attains a minimum value
of $\gamma - (\gamma-1)\log(\gamma-1)$ at $q = (\gamma-1)/\gamma$.
The term $(\gamma-1)\log\frac{p}{1-p}$
of Eq.~\eqref{eq:h_p''} is increasing for $p \in (0,1)$ and
surjective onto the reals. Therefore, $\Ip''(x^{1/d}) \geq 0$ for all $x$ if
$\gamma - (\gamma-1)\log(\gamma-1) + (\gamma-1) \log \frac{p}{1-p} \geq 0$, which is
equivalent to having $p \geq \frac{\gamma-1}{\gamma-1 + e^{\gamma/(\gamma-1)}}$. Additionally, if $p <
\frac{\gamma-1}{\gamma-1 + e^{\gamma/(\gamma-1)}}$, then $\Ip''(x)$ starts as positive, becomes
  negative, then turns positive again.
\end{proof}

\begin{figure}
\begin{center}
  \begin{tikzpicture}[font=\small]
    \node[anchor=south west] (plot) at (0,0){
         \includegraphics{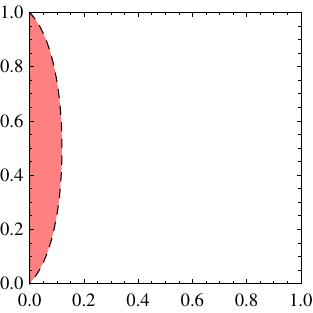}};
    \begin{scope}[x={(plot.south east)},y={(plot.north west)}]
    \node at (.17,.53) {$\cB_2$};
  \end{scope}
  \end{tikzpicture}
\end{center}
\caption{The region $\cB_2$ consisting of all points $(p,q)$ such that $(q^2,\Ip(q))$ does not lie on the convex minorant of $x\mapsto \Ip(\sqrt{x})$.}
\label{fig:B2-phase}
\end{figure}

We next give an explicit description of the
region $\cB_2$.

\begin{lemma}[$\gamma = 2$ case]
  \label{lem:d2-boundary}
  Let $p, q \in (0,1)$. The point $(q^2, \Ip(q))$ lies strictly above
  the
  convex minorant of $x \mapsto \Ip(\sqrt{x})$ if and only if
  $p < \paren{1 + (q^{-1} - 1)^{1/(1-2q)}}^{-1}$.
\end{lemma}

\begin{proof}
We claim that the lower
common tangent of $x \mapsto \Ip(\sqrt x)$ has slope $\log(\frac{1-p}{p})$. To show this, it
suffices to check that $\Ip(\sqrt x) - x \log(\frac{1-p}{p})$ has a
horizontal common lower tangent, and it suffices to check the same
thing for $\Ip(x) - x^2 \log (\frac{1-p}{p})$. Observe that
\begin{align*}
  \Ip(x) - x^2\log\frac{1-p}p
  &= x \log \frac{x}{p} + (1-x)\log \frac{1-x}{1-p} -
   x^2\log\frac{1-p}{p}
  \\
  &= x\log x + (1-x)\log(1-x)-x(1-x)\log\frac{p}{1-p} - \log(1-p)
\end{align*}
is invariant under $x \mapsto 1-x$, so that its lower tangent must be
horizontal by symmetry, and let it touch the curve at $x = \uq, \oq$,
so that $0 < \uq < \oq < 1$ are the zeros of the derivative, namely
\begin{equation}
  \label{eq:tangent-deriv}
\log\frac{x}{1-x} - (1-2x)\log\frac{p}{1-p}\,.
\end{equation}
It follows that $(q^2, \Ip(q))$ lies strictly above the convex
minorant if and only if $\uq < q < \oq$, which is equivalent to
having $\frac{1}{1-2q}\log\frac{q}{1-q} \leq \log\frac{p}{1-p}$.
Rearranging the latter concludes the proof.
\end{proof}

For $\gamma > 1$ and $0 < p < p_0(\gamma)$, define $\uq = \uq(\gamma,p)$ and $\oq = \oq(\gamma,p)$ to be such that the
lower common tangent to $x \mapsto \Ip(x^{1/\gamma})$ touches the curve at points
$(\uq^\gamma, \Ip(\uq)$ and $(\oq^\gamma, \Ip(\oq))$.
An examination of the geometry of the curve, as
  illustrated in Fig.~\ref{fig:uq-q1-q2-oq}, immediately leads to the following lemma:

\begin{lemma}
  \label{lem:hp-contain}
  Let $\gamma > 1$ and $0 < p < p_0(\gamma)$. Let $0 < q_1 < q_2 < 1$. If the line
  segment joining points $(q_1^\gamma, \Ip(q_1))$ and $(q_2^\gamma, \Ip(q_2))$
  lies below the curve $\{(q^\gamma, \Ip(q)) : 0 \leq q \leq 1\}$ and is not
  tangent to the curve at one of the end points, then this segment lies
  strictly above the lower common tangent of the curve. Consequently,
  $\uq (\gamma, p) < q_1 < q_2 < \oq (\gamma, p)$.
\end{lemma}

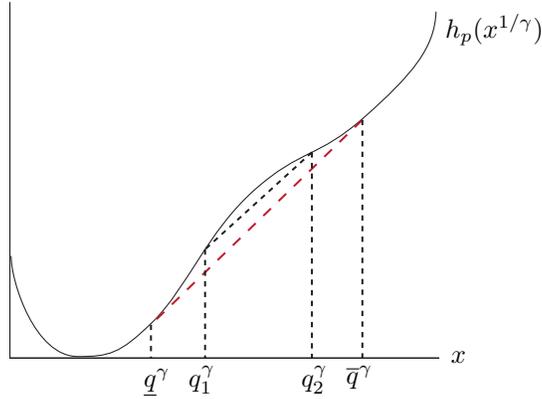
\begin{figure}
\begin{center}
  \begin{tikzpicture}[font=\small]
    \node[anchor=south west] (plot) at (0,0)
    {\includegraphics{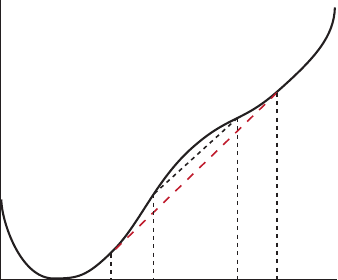}};
    \begin{scope}[x={(plot.south east)},y={(plot.north west)}]
      \path[use as bounding box] (0,-.08) rectangle (1.23,1);
    \node at (.35,-.038) {$\uq^\gamma$};
    \node at (.45,-.03) {$q_1^\gamma$};
    \node at (.7,-.03) {$q_2^\gamma$};
    \node at (.8,-.024) {$\oq^\gamma$};
    \node at (1.02,.03) {$x$};
    \node at (1.1,.9) {$\Ip(x^{1/\gamma})$};
    \end{scope}
  \end{tikzpicture}
  \vspace{-0.3cm}
\end{center}
\caption{Illustration of the convex minorant of $x\mapsto \Ip(x^{1/\gamma})$ in the setting of Lemma~\ref{lem:hp-contain}.}
\label{fig:uq-q1-q2-oq}
  \end{figure}

We now apply the above lemma to describe the shape of the
regions $\cB_\gamma$.

\begin{lemma}
  \label{lem:hp-p}
  If $\gamma > 1$ and $0 < p < p' < p_0(\gamma)$, then $\uq(\gamma, p) < \uq(\gamma,
  p') < \oq(\gamma, p') < \oq(\gamma, p)$. So $\cB_\gamma$ is a
  rotated-V-shaped region.
\end{lemma}

  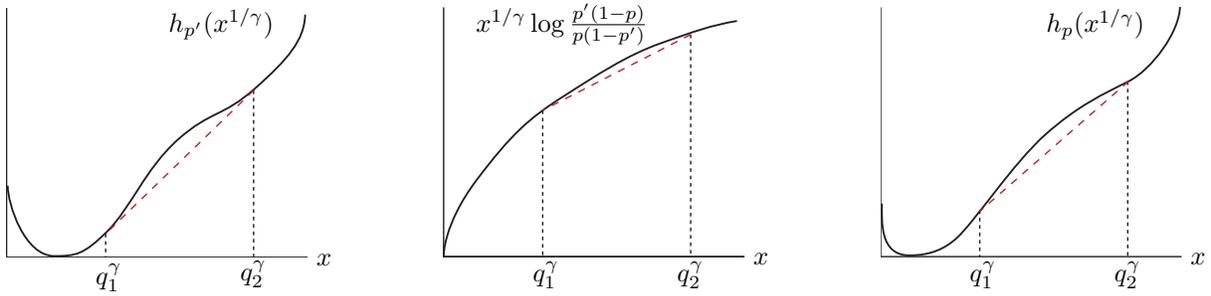
\begin{figure}
  \begin{center}
  \begin{tikzpicture}[font=\small]
    \node[] (plot1) at (0,0)
    {\includegraphics[scale=.7]{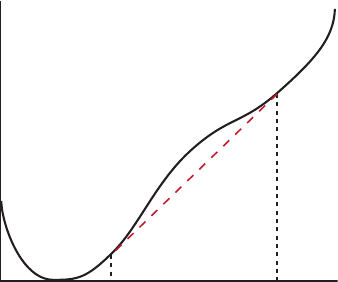}}; 
    \node[right=4em] (plot2) at (plot1.east)
    {\includegraphics[scale=.7]{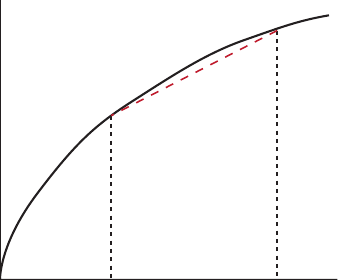}};
    \node[right=4em] (plot3) at (plot2.east)
    {\includegraphics[scale=.7]{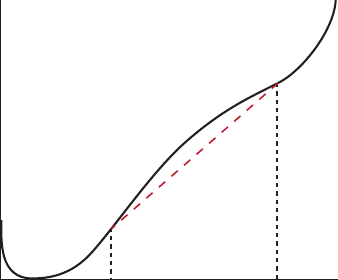}};


    \begin{scope}[shift={(plot1.south west)},x={(plot1.south east)},y={(plot1.north west)}]
    \node at (.35,-.038) {$q_1^\gamma$};
    \node at (.8,-.024) {$q_2^\gamma$};
    \node at (1.02,.03) {$x$};
    \node at (.7,.9) {$h_{p'}(x^{1/\gamma})$};
    \end{scope}
    \begin{scope}[shift={(plot2.south west)},x={(plot2.south east)},y={(plot2.north west)}]
    \node at (.35,-.03) {$q_1^\gamma$};
    \node at (.8,-.03) {$q_2^\gamma$};
    \node at (1.02,.03) {$x$};
    \node at (.4,.9) {$x^{1/\gamma} \log \frac{p'(1-p)}{p(1-p')}$};
    \end{scope}
    \begin{scope}[shift={(plot3.south west)},x={(plot3.south east)},y={(plot3.north west)}]
    \node at (.35,-.03) {$q_1^\gamma$};
    \node at (.8,-.03) {$q_2^\gamma$};
    \node at (1.02,.03) {$x$};
    \node at (.7,.9) {$h_{p}(x^{1/\gamma})$};
    \end{scope}
  \end{tikzpicture}
    \vspace{-0.3cm}
\end{center}
\caption{Illustration of the functions in Eq.~\eqref{eq-hp(x^(1/gamma))-breakdown}.}
\label{fig:hp(x^(1/gamma))-breakdown}
\end{figure}

\begin{proof}
  Let $q_1 = \uq(\gamma, p')$ and $q_2 = \oq(\gamma, p')$.
   As illustrated in Fig.~\ref{fig:hp(x^(1/gamma))-breakdown}, we have
  \begin{equation}
    \label{eq-hp(x^(1/gamma))-breakdown}
  \Ip(x^{1/\gamma}) = h_{p'}(x^{1/\gamma}) + x^{1/\gamma} \log\frac{p'(1-p)}{p(1-p')} +
  \log\frac{1-p'}{1-p} \, .
  \end{equation}
  The segment joining $(q_1^\gamma,h_{p'}(q_1))$ and $(q_2^\gamma, h_{p'}(q_2))$
  lies below $x \mapsto h_{p'}(x^{1/\gamma})$ by definition.
  Since $x \mapsto x^{1/\gamma} \log\frac{p'(1-p)}{p(1-p')}$ is concave,
  the segment joining $x = q_1^\gamma$ and $x = q_2^\gamma$ must also lie below
  the curve $x \mapsto \Ip(x^{1/\gamma})$, and it is not tangent to either
  endpoint due to the $x^{1/\gamma}$ term. The conclusion now follows from Lemma~\ref{lem:hp-contain}.
\end{proof}

\begin{lemma}
  \label{lem:hp-d}
  If $\gamma > \gamma' > 1$ and $0 < p< p_{\gamma'}$, then $\uq(\gamma, p) < \uq(\gamma',
  p) < \oq(\gamma', p) < \oq(\gamma, p)$. So $\cB_\gamma$ strictly
  contains $\cB_{\gamma'}$.
\end{lemma}

\begin{proof}
  Let $q_1 = \uq(\gamma', p)$ and $q_2 = \uq(\gamma', p)$. Let $\ell'$
  denote the line segment joining points $(q_1^{\gamma'}, \Ip(q_1))$ and
  $(q_2^{\gamma'}, \Ip(q_2))$, so that $\ell'$ is tangent to the curve $x
  \mapsto \Ip(x^{1/\gamma'})$.

  Consider a transformation of the plots induced by the change of
  variable $x = u^{\gamma'/\gamma}$. The plot of $x \mapsto
  \Ip(x^{1/\gamma'})$ becomes the plot of $u \mapsto
  \Ip(u^{1/\gamma})$ (see Fig.~\ref{fig:x-(1/gamma)-u-(1/gamma)}). Originally $\ell'$ was a line segment of positive
  slope lying below the curve $x \mapsto \Ip(x^{1/\gamma'})$ and
  tangent to it at both endpoints. Following the transformation $x =
  u^{\gamma'/\gamma}$ (recall that $\gamma'/\gamma < 1$) we see that $\ell'$
  becomes a concave curve $\ell$ still lying below the new curve $u
  \mapsto \Ip(u^{1/\gamma})$ and tangent to it at both endpoints.
  This implies that the line segment joining the two endpoints of $\ell$
  in the new frame lies below both curves and is tangent to neither at
  the endpoints. The conclusion then follows from
  Lemma~\ref{lem:hp-contain}.
 \begin{figure}
  \begin{center}
  \begin{tikzpicture}[font=\small]
    \node[] (plot1) at (0,0)
    {\includegraphics[scale=.85]{figures/fig-minorant-curve-a}};
    \node[right=6em] (plot2) at (plot1.east)
    {\includegraphics[scale=.85]{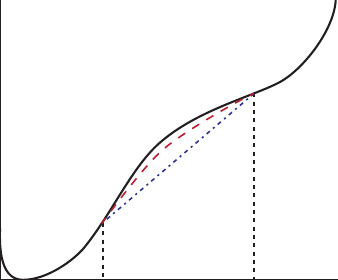}};

    \node[shift={(1,0)}] at (plot1.east) {\normalsize$\xrightarrow{x
        = u^{\gamma'/\gamma}}$};

    \begin{scope}[shift={(plot1.south west)},x={(plot1.south east)},y={(plot1.north west)}]
    \node at (.35,-.03) {$q_1^{\gamma'}$};
    \node at (.8,-.03) {$q_2^{\gamma'}$};
    \node at (1.02,.03) {$x$};
    \node at (.7,.9) {$h_{p}(x^{1/{\gamma'}})$};
    \end{scope}
    \begin{scope}[shift={(plot2.south west)},x={(plot2.south east)},y={(plot2.north west)}]
    \node at (.32,-.03) {$q_1^\gamma$};
    \node at (.75,-.03) {$q_2^\gamma$};
    \node at (1.02,.03) {$u$};
    \node at (.4,.9) {$\Ip(u^{1/\gamma})$};
    \end{scope}
  \end{tikzpicture}
\end{center}
\caption{The plot of $x\mapsto \Ip(x^{1\gamma})$ following a change of variable $x=u^{\gamma'/\gamma}$.}
\label{fig:x-(1/gamma)-u-(1/gamma)}
\end{figure}
\end{proof}

\bigskip
\section*{Acknowledgments}

We thank Amir Dembo and Ofer Zeitouni for fruitful discussions.
This work was initiated while Y.\ Z.\ was an intern at the Theory Group of Microsoft Research, and
he thanks the Theory Group for its hospitality.

\bibliographystyle{abbrv}
\bibliography{ldp_ref}

\end{document}